\documentclass[leqno,preprint]{imsart}

\pdfoutput=1

\usepackage[dvips,letterpaper,hcentering,vcentering,textwidth=16.5cm]{geometry}
\usepackage{amsmath,amssymb,amsthm,amsbsy,amsfonts}
\usepackage{enumerate}
\usepackage{subfigure}
\usepackage{pdfpages}

\renewcommand{\le}{\leqslant}

\renewcommand{\P}{\operatorname{\mathsf{P}}}

\DeclareMathOperator{\E}{\mathsf{E}}
\DeclareMathOperator{\var}{\mathsf{Var}}
\DeclareMathOperator{\cov}{\mathsf{Cov}}

\DeclareMathOperator{\are}{ARE}

\DeclareMathOperator{\sign}{sign}
\newcommand{\asin}{\operatorname{sin}^{-1}}

\providecommand{\abs}[1]{\lvert#1\rvert}
\providecommand{\bigabs}[1]{\bigl\lvert#1\bigr\rvert}
\providecommand{\Bigabs}[1]{\Bigl\lvert#1\Bigr\rvert}

\providecommand{\norm}[1]{\lVert#1\rVert}

\providecommand{\Bigprob}[2][]{\P_{#1}\Bigl(#2\Bigr)}

\newcommand{\ind}[1]{\operatorname{\mathsf{I}}\{#1\}}

\newcommand{\inc}{\nearrow}
\newcommand{\dec}{\searrow}

\newcommand{\conv}{\mathop{\longrightarrow}\limits}

\newcommand{\tr}{\ensuremath{^\mathsf{T}}} 

\renewcommand{\th}{\theta}
\newcommand{\Th}{\Theta}

\providecommand{\si}{\sigma}
\providecommand{\Si}{\Sigma}

\newcommand{\R}{\mathbb{R}}

\newcommand{\N}{\mathbb{N}}

\newcommand{\V}{\mathcal{V}}

\newtheorem{thm}{Theorem}[section]
\newtheorem*{thm*}{Theorem}
\newtheorem*{thmA}{Theorem A: Special-case (l'Hospital-type monotonicity) rules}
\newtheorem*{thmB}{Theorem B: General (l'Hospital-type monotonicity) rules}
\newtheorem*{thmC}{Theorem C: Refined general (l'Hospital-type monotonicity) rules}
\newtheorem{cor}[thm]{Corollary}
\newtheorem{lem}[thm]{Lemma}

\theoremstyle{remark}

\newtheorem{remark}[thm]{Remark}
\newtheorem*{remark*}{Remark}

\numberwithin{equation}{section}
\numberwithin{thm}{section}

\begin{document}

\begin{frontmatter}

\title{Monotonicity properties of the asymptotic relative efficiency between common correlation statistics in the bivariate normal model}
\runtitle{Monotonicity properties of the ARE}

\begin{aug}
  \author{\fnms{Raymond} \snm{Molzon}\thanksref{a}}
  \and  
  \author{\fnms{Iosif}  \snm{Pinelis}\thanksref{a,e1}\ead[label=e1,mark]{ipinelis@mtu.edu}}
  \runauthor{Raymond Molzon and Iosif Pinelis}
  \affiliation{Michigan Technological University}
  \address[a]{Department of Mathematical Sciences\\Michigan Technological University\\Hough\-ton, Michigan 49931\\\printead{e1}}
\end{aug}

\begin{abstract}
Pearson's $R$ is the most common correlation statistic, used mainly in parametric settings. Most common among nonparametric correlation statistics are Spearman's $S$ and Kendall's $T$. 
We show that for bivariate normal i.i.d.\ samples the pairwise asymptotic relative efficiency ($\are$) between these three statistics depends monotonically on the population correlation coefficient $\rho$. 
Namely, $\are_{R,T}(|\rho|)$ increases in $|\rho|\in[0,1)$ from $1.096\dotsc$ to $1.209\dotsc$, $\are_{R,S}(|\rho|)$ increases from $1.096\dotsc$ to $1.439\dotsc$ and $\are_{T,S}(|\rho|)$ increases from 1 to $1.190\dotsc$. 
This monotonicity is a corollary to a stronger result, which asserts that $q_{R,T;a}(\rho)=(\are_{R,T}(\rho)-\are_{R,T}(a)-\are_{R,T}'(a)(\rho-a))/(\rho-a)^2$ is increasing in $\rho\in(0,1)$, where $a\in\{0,1\}$, and similarly for the two functions $q_{R,S;a}$ and $q_{T,S;a}$ with $\are_{R,S}$ and $\are_{T,S}$ in place of $\are_{R,T}$. 
Another immediate corollary is the existence of quadratic (in $\rho$) polynomials $L_{R,T;a}(\rho)$ and $U_{R,T;a}(\rho)$ such that $L_{R,T;a}(\rho)\le\are_{R,T}(\rho)\le U_{R,T;a}(\rho)$ for $\rho\in(0,1)$ and $a\in\{0,1\}$, and similar quadratic bounds are given for $\are_{R,S}$ and $\are_{T,S}$. 
The proofs rely on the use of l'Hospital-type rules for monotonicity patterns.
\end{abstract}



\end{frontmatter}

\tableofcontents

\section{Introduction}\label{sec:intro}
Pearson's $R$, Spearman's $S$ and Kendall's $T$ are the three most commonly used correlation statistics, the latter two especially in nonparametric studies. 
When the population distribution is bivariate normal, the question of independence between the two random variables (r.v.'s) reduces to deciding if the population correlation $\rho$ is 0. 
In the case of testing $H_0\colon\rho=0$, it is known that the Pitman asymptotic relative efficiency ($\are$) of $R$ to $S$ is $\frac{\pi^2}9$ \cite{hot36} and that of $T$ to $S$ is 1 \cite{moran51} (and hence the $\are$ of $R$ to $T$ is $\frac{\pi^2}9$ as well).
While perhaps less common in practice, one could also use any three of these statistics to test hypotheses of the form $H_0\colon\rho=\rho_0$ (against alternatives $\rho>\rho_0$, $\rho<\rho_0$, or $\rho\ne\rho_0$) for arbitrary $\rho_0\in(-1,1)$. 
In \cite{bor02}, values of the $\are_{S,R}(\rho_0)$ (the $\are$ of $S$ to $R$ for the null hypothesis $\rho=\rho_0$) are tabulated for several values of $\rho_0\in[0,1)$; several values of $\are_{T,R}(\rho_0)$ are given in \cite{moran79} as well. 

In this paper, we show that $\are_{R,T}(|\rho_0|)$ is strictly increasing in $|\rho_0|\in[0,1)$ from $1.096\dotsc$ to $1.209\dotsc$, $\are_{R,S}(|\rho_0|)$ increases from $1.096\dotsc$ to $1.439\dotsc$, and $\are_{T,S}(|\rho_0|)$ increases from 1 to $1.190\dotsc$. 
Thus, all these $\are$'s stay rather close to 1 for all values of $\rho_0\in(-1,1)$. 
Additionally, several upper and lower quadratic bounds are shown to take place for each of $\are_{R,T}$, $\are_{R,S}$ and $\are_{T,S}$. 
All of these results are immediate corollaries to a stronger result, stated in this paper as Theorem~\ref{thm:quad}.

For testing $H_0\colon\th=\th_0$ in the framework of a given statistical model (against any of the alternative hypotheses $\th\ne\th_0$, $\th>\th_0$, or $\th<\th_0$), under certain general conditions there exists an easily applicable formula for computing the $\are$ between two (sequences of) real-valued test statistics $T_1=(T_{1,n})_{n\in\N}$ and $T_2=(T_{2,n})_{n\in\N}$. 
The main condition (see e.g.~\cite{noe55,hoeff55,konijn56}) is that the distribution function (d.f.)\ of either properly normalized test statistic converges to the standard normal d.f.\ $\Phi$ uniformly in a certain sense as the sample size $n$ tends to $\infty$. 
Particularly, if there exist continuous real-valued functions $\mu_{T_j}$ and $\si_{T_j}$ on the parameter space $\Th$ such that
\begin{equation}\label{eq:ub}
\sup_{\th\in\V}\sup_{z\in\R}
  \Bigabs{\Bigprob[\th]{\frac{T_{j,n}-\mu_{T_j}(\th)}{\sigma_{T_j}(\th)/\sqrt n}\le z}-\Phi(z)}
  \conv_{n\rightarrow\infty}0,
\end{equation}
where $\V$ is some neighborhood of $\th_0$ chosen such that $\mu_{T_j}$ is continuously differentiable and $\si_{T_j}>0$ on $\V$ for $j=1,2$, then the $\are$ of $T_1$ to $T_2$ may be expressed by the formula
\begin{equation}\label{eq:are}
\are(\th_0):=\are_{T_1,T_2}(\th_0)
  =\frac{\si_{T_2}^2(\th_0)}{\si_{T_1}^2(\th_0)}\,\frac{\mu_{T_1}'(\th_0)^2}{\mu_{T_2}'(\th_0)^2},
\end{equation}
assuming that $\mu_{T_j}'(\th_0)>0$. 
The functions $\mu_{T_j}$ and $\si_{T_j}/\sqrt n$ may be called the asymptotic mean and standard deviation, respectively, of the sequence $T_j$. 

Berry-Ess\'een bounds provide a nice way to verify the condition \eqref{eq:ub}. 
Such bounds for the Kendall and Spearman statistics, which are instances of so-called $U$- and $V$-statistics, are essentially well known; see e.g.\ \cite{kor94}; in fact, we are using here a result by Chen and Shao \cite{chen07} and a convenient representation of any $V$-statistic as a $U$-statistic \cite{hoeff63}. 
As for a Berry-Ess\'een bound for the Pearson correlation statistic, we are using an apparently previously unknown result in \cite{nonlinear}. 

According to the formula \eqref{eq:are}, the ARE between two test statistics can be expressed in terms of the asymptotic means and variances of the two statistics. 
In turn, the asymptotic variance of either $T$ or $S$ in the bivariate normal model 
can be expressed using Schl\"afli's formula \cite{schlaf60} for the volume of the spherical tetrahedron in $\R^4$. 
Such formulas have been of significant interest to a number of authors; see e.g.\ the recent papers \cite{kohno07} and \cite{mura05}. 
We remark also that Plackett \cite{plack54} obtained a result more general than Schl\"afli's. 
Actually, here we are using formulas by David and Mallows \cite{david61} which are based on \cite{plack54}. 

To prove the main result, we use l'Hospital-type rules for the monotonicity pattern of a function $r=\frac{f}{g}$ on some interval $(a,b)$. 
Knowledge of the monotonicity of $\frac{f'}{g'}$ on $(a,b)$, along with the sign of $gg'$ on $(a,b)$, allows one to obtain the monotonicity pattern of $r$; see Pinelis \cite{pin06} and the bibiliography there for several variants of these rules and applications to various problems. 
For convenient reference these rules are stated as Theorems~A, B, and C in Section~\ref{sec:proof.mon}.

\section{Monotonicity properties of the ARE in the bivariate normal model}\label{sec:mon}
Let $(V_n)=:\bigl((X_n,Y_n)\bigr)$ be a sequence of independent, identically distributed (i.i.d.) nondegenerate bivariate normal r.v.'s with 
\begin{equation*}
\E V_1=:(\mu_X,\mu_Y)\quad\text{and}\quad
\cov(V_1)=:\begin{bmatrix}\sigma_X^2&\rho\sigma_X\sigma_Y\\\rho\sigma_X\sigma_Y&\sigma_Y^2\end{bmatrix};
\end{equation*}
note that $\rho\in(-1,1)$ is the correlation coefficient between $X_1$ and $Y_1$. 
Let
\begin{equation}\label{eq:R}
R:= R_n:= \frac{\sum_{i=1}^n(X_i-\bar X)(Y_i-\bar Y)}
  {\sqrt{\sum_{i=1}^n(X_i-\bar X)^2}\sqrt{\sum_{i=1}^n(Y_i-\bar Y)^2}},
\end{equation}
where $(\bar X,\bar Y):= \frac1n\sum_{i=1}^nV_i$; 
$R$ is commonly called Pearson's product-moment correlation coefficient, and it is the maximum-likelihood estimator of $\rho$. 
Spearman's rank correlation is
\begin{equation}\label{eq:S}
S:= S_n
  := \frac{12}{n^3-n}\sum_{i=1}^nr(X_i)r(Y_i)-\frac{3(n+1)}{n-1},
\end{equation}
where $r(X_i):=\sum_{j=1}^n\ind{X_j\le X_i}$ and $r(Y_i):=\sum_{k=1}^n\ind{Y_k\le Y_i}$ are the ranks (and $\ind{\cdot}$ denotes the indicator function). 
Note that $S$ is simply the product-moment correlation of the sample of ranks $\bigl(r(X_1)$,$r(Y_1)),\dotsc$, $(r(X_n),r(Y_n)\bigr)$. 
Let next
\begin{equation*}
J_{ij}:= \ind{X_j<X_i}\ind{Y_j<Y_i},
\end{equation*}
and let
\begin{equation}\label{eq:T}
T:= T_n:= \frac1{\binom n2}\sum_{1\le i<j\le n}h_T(V_i,V_j)
\end{equation}
denote Kendall's correlation statistic, where $h_T(V_i,V_j):= 2(J_{ij}+J_{ji})-1$, so that almost surely (a.s.) $h_T(V_i,V_j)=\pm1$ depending on whether the pair $(V_i,V_J)$ is concordant or discordant; also, $\E h_T(V_i,V_j)=0$ if $V_i$ and $V_j$ are independent.

Consider the hypothesis test $H_0\colon\rho=\rho_0$ against the alternative $H_1\colon\rho\ne\rho_0$ (or again, either of the two one-sided alternatives), where $\rho_0\in(-1,1)$. 
We shall show that each of $R$, $S$, and $T$ satisfies the condition \eqref{eq:ub}, so that \eqref{eq:are} may be used to express the $\are$ between any two of these statistics. 
Further, it is easy to see, and also will be clear from what follows, that $\si_R^2$, $\si_S^2$, and $\si_T^2$ are all even functions of $\rho$, and also $\mu_R$, $\mu_S$, and $\mu_T$ are odd functions, so that the $\are$ of any pair of these statistics is even. 
See Figure~\ref{fig:are} for a plot of these three functions, and note it suggests each of the pairwise $\are$'s is strictly increasing on $(0,1)$. 
Further, the shapes of these plots suggest the functions may be well-approximated by a quadratic polynomial. 
Indeed, the monotonicity of the $\are$ and a quadratic approximation shall be immediate results of the following:

\begin{thm}\label{thm:quad}
For the test of the null hypothesis $\rho=\rho_0$ against any of the three alternative hypotheses: $\rho\ne\rho_0$, $\rho>\rho_0$, or $\rho<\rho_0$, let
\begin{equation*}
q_a(\rho_0):= q_{T_1,T_2;a}(\rho_0):=
  \frac{\are_{T_1,T_2}(\rho_0)-\are_{T_1,T_2}(a)-\are_{T_1,T_2}'(a)(\rho_0-a)}{(\rho_0-a)^2}
\end{equation*}
for $\rho_0\in[0,a)\cup(a,1)$ and $a\in[0,1]$, where $T_i$ is one of the statistics $R$, $S$, or $T$; 
here and in what follows, $\are_{T_1,T_2}(a)$ and $\are_{T_1,T_2}'(a)$ are understood to mean $\are_{T_1,T_2}(1-)$ and $\are_{T_1,T_2}'(1-)$, respectively, when $a=1$. 
Then
\begin{enumerate}
\item[(RT0)]\label{quadRT0}
$q_{R,T;0}$ is increasing from 0.0966$\dotsc$ to 0.1125$\dotsc$;
\item[(RT1)]\label{quadRT1}
$q_{R,T;1}$ is increasing from 0.1510$\dotsc$ to 0.2247$\dotsc$;
\item[(TS0)]\label{quadTS0}
$q_{T,S;0}$ is increasing from 0.0984$\dotsc$ to 0.1904$\dotsc$;
\item[(TS1)]\label{quadTS1}
$q_{T,S;1}$ is increasing from 0.5516$\dotsc$ to 1.8200$\dotsc$;
\item[(RS0)]\label{quadRS0}
$q_{R,S;0}$ is increasing from 0.2045$\dotsc$ to 0.3428$\dotsc$;
\item[(RS1)]\label{quadRS1}
$q_{R,S;1}$ is increasing from 0.8682$\dotsc$ to 2.6639$\dotsc$.
\end{enumerate}
\end{thm}
The term ``increasing'' will mean for us ``strictly increasing,'' and similarly ``decreasing'' will mean ``strictly decreasing.'' 
Exact expressions for the endpoint values 0.0966$\dotsc$, 0.1125$\dotsc$, $\dotsc$ are given at the end of the respective sections RT0, RT1, \dots in the appendices. 
All proofs are deferred to Section~\ref{sec:proofs}.

\begin{figure}[h!t]
\begin{center}
\includegraphics[height=5cm]{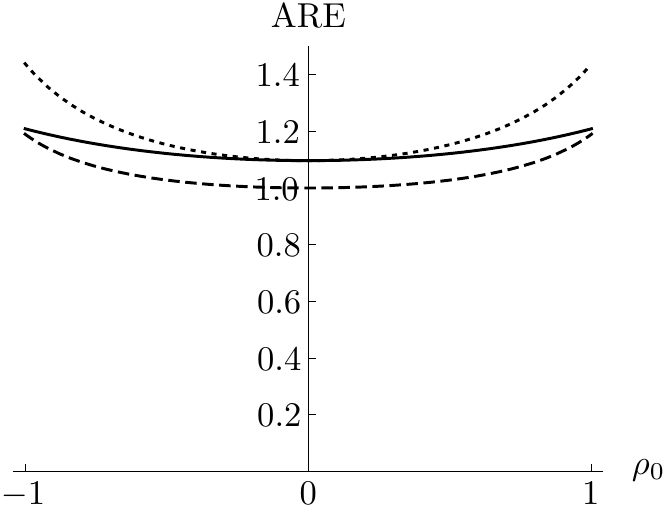}
\end{center}
\caption{Plots of $\are_{R,T}(\rho_0)$ (solid), $\are_{T,S}(\rho_0)$ (dashed) and $\are_{R,S}(\rho_0)$ (dotted)}
\label{fig:are}
\end{figure}

\begin{cor}\label{cor:mon}
For the test of the null hypothesis $\rho=\rho_0$ against any of the three alternative hypotheses: $\rho\ne\rho_0$, $\rho>\rho_0$, or $\rho<\rho_0$, one has
\begin{enumerate}
\item[(RT)]\label{areRT} $\are_{R,T}(|\rho_0|)$ is increasing in $|\rho_0|\in(0,1)$ from $\frac{\pi^2}9=\text{1.0966}\dotsc$ to $\frac{2\pi\sqrt3}{9}=\text{1.2091}\dotsc$;
\item[(TS)]\label{areTS} $\are_{T,S}(|\rho_0|)$ is increasing in $|\rho_0|\in(0,1)$ from 1 to $\frac{9\sqrt3(11\sqrt5-15)}{40\pi}=\text{1.1904}\dotsc$;
\item[(RS)]\label{areRS} $\are_{R,S}(|\rho_0|)$ is increasing in $|\rho_0|\in(0,1)$ from $\frac{\pi^2}9=\text{1.0966}\dotsc$ to $\frac{3(11\sqrt5-15)}{20}=\text{1.4395}\dotsc$.
\end{enumerate}
\end{cor}
This corollary justifies the conjecture that the pairwise $\are$'s are increasing on $(0,1)$, which one would make from observing Figure~\ref{fig:are}.

\begin{cor}\label{cor:bounds}
Let
\begin{align*}
L_a(x):=L_{T_1,T_2;a}(x)&:=\are_{T_1,T_2}(a)+\are_{T_1,T_2}'(a)(|x|-a)+q_{T_1,T_2;a}(0+)(|x|-a)^2,\\
U_a(x):=U_{T_1,T_2;a}(x)&:=\are_{T_1,T_2}(a)+\are_{T_1,T_2}'(a)(|x|-a)+q_{T_1,T_2;a}(1-)(|x|-a)^2,\\
L(x):=L_{T_1,T_2}(x)  &:=L_{T_1,T_2;0}(x)\vee L_{T_1,T_2;1}(x),\\
\text{and}\quad
U(x):=U_{T_1,T_2}(x)  &:=U_{T_1,T_2;0}(x)\wedge U_{T_1,T_2;1}(x)
\end{align*}
for $(T_1,T_2)\in\{(R,T),(T,S),(R,S)\}$, $x\in(-1,1)$, and $a\in\{0,1\}$. 
Then for all $\rho_0\in(-1,0)\cup(0,1)$
\begin{equation*}
L_a(\rho_0)\le L(\rho_0)<\are(\rho_0)<U(\rho_0)\le U_a(\rho_0).
\end{equation*}
\end{cor}

These piecewise quadratic bounds are illustrated in Figure~\ref{fig:bounds}. Note that $L_0$ and $U_0$ give good quadratic approximations to the $\are$ near the origin, while $L_1$ and $U_1$ are better approximations when $\rho_0$ is near $\pm1$. 

\begin{figure}[h!t]
\subfigure[Plots of $f_{T_1,T_2}$ for $f=\are$ (solid), $f=L$ (dotted) and $f=U$ (dashed)]{
\begin{tabular}{*{3}{c@{\hspace{.5cm}}}}
\includegraphics[width=4.5cm]{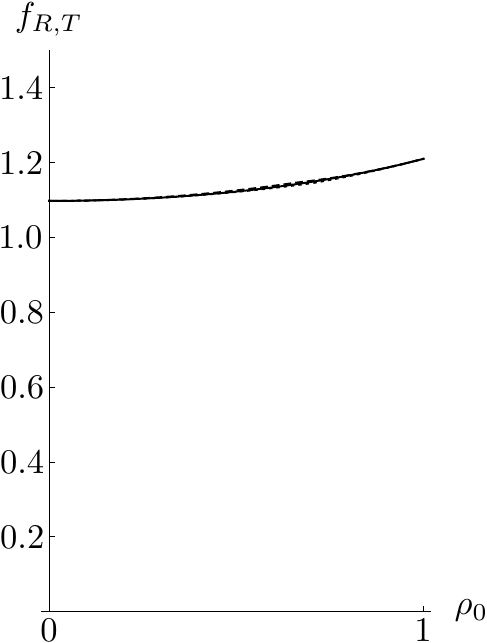}
&\includegraphics[width=4.5cm]{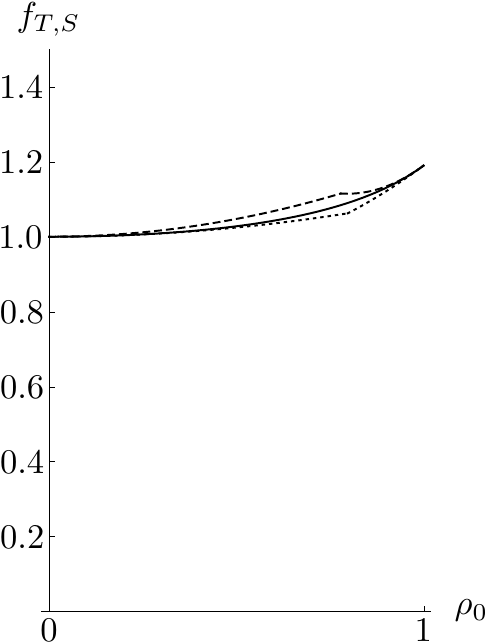}
&\includegraphics[width=4.5cm]{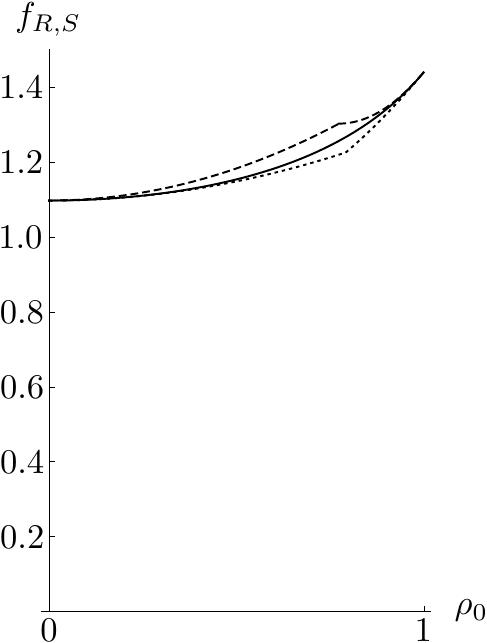}
\end{tabular}
\label{fig:bounds1}
}
\\
\subfigure[Rescaled plots of $f_{T_1,T_2}$ for $f=\are$ (solid), $f=L$ (dotted) and $f=U$ (dashed)]{
\begin{tabular}{*{3}{c@{\hspace{.5cm}}}}
\includegraphics[width=4.5cm]{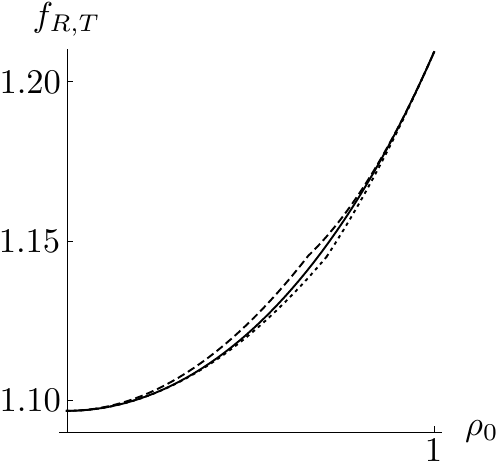}
&\includegraphics[width=4.5cm]{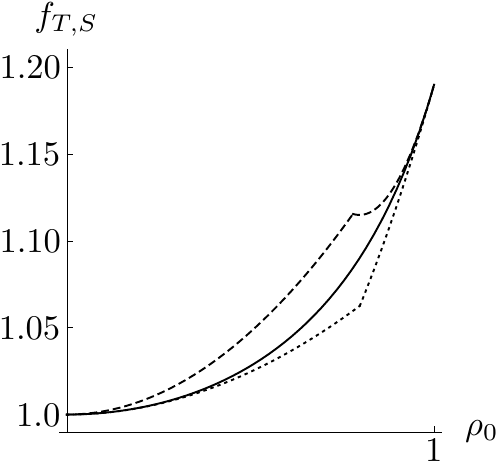}
&\includegraphics[width=4.5cm]{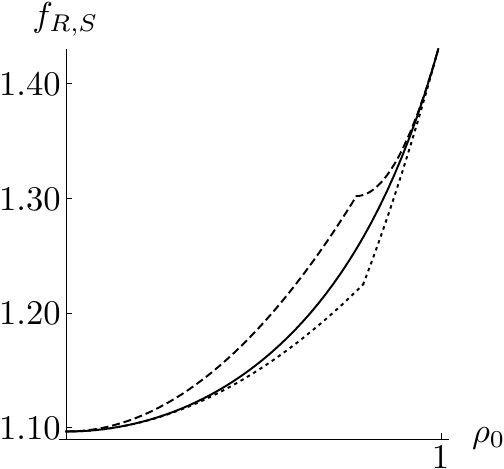}
\end{tabular}
\label{fig:bounds.scaled}
}
\label{fig:bounds}
\caption{Illustration of piecewise quadratic bounds of Corollary~\ref{cor:bounds}}
\end{figure}

\begin{remark}\label{re:num}
Numerical approximations to the various bounds in Corollary~\ref{cor:bounds} are given below. 
\begin{align*}
 L_{R,T;0}(x)&\approx1.0966+0.0966x^2;
&L_{R,T;1}(x)&\approx1.0966-0.0384|x|+0.1510x^2;\\
 U_{R,T;0}(x)&\approx1.0966+0.1126x^2;
&U_{R,T;1}(x)&\approx1.1704-0.1860|x|+0.2248x^2;\\
 L_{T,S;0}(x)&\approx1+0.0984x^2;
&L_{T,S;1}(x)&\approx1-0.3612|x|+0.5516x^2;\\
 U_{T,S;0}(x)&\approx1+0.1905x^2;
&U_{T,S;1}(x)&\approx2.2684-2.8980|x|+1.8200x^2;\\
 L_{R,S;0}(x)&\approx1.0966+0.2046x^2;
&L_{R,S;1}(x)&\approx1.0966-0.5254|x|+0.8683x^2;\\
 U_{R,S;0}(x)&\approx1.0966+0.3429x^2;
&U_{R,S;1}(x)&\approx2.8924-4.1169|x|+2.6640x^2.
\end{align*}
Further, one has
\begin{align*}
 L_{R,T;0}(x)=L_{R,T;1}(x)\text{ when }x\approx0.7067;
&\quad U_{R,T;0}(x)=U_{R,T;1}(x)\text{ when }x\approx0.6573;\\
 L_{T,S;0}(x)=L_{T,S;1}(x)\text{ when }x\approx0.7969;
&\quad U_{T,S;0}(x)=U_{T,S;1}(x)\text{ when }x\approx0.7784;\\
 L_{R,S;0}(x)=L_{R,S;1}(x)\text{ when }x\approx0.7916;
&\quad U_{R,S;0}(x)=U_{R,S;1}(x)\text{ when }x\approx0.7737.
\end{align*}
\end{remark}

\begin{remark}\label{re:tighter}
We note that piecewise quadratic bounds even tighter than the $L_{T_1,T_2}$ and $U_{T_1,T_2}$ could be obtained from Theorem~\ref{thm:quad}. 
The bounds on the $\are$ given in Corollary~\ref{cor:bounds} are derived by appropriately rewriting the inequalities $q_a(0+)<q_a(x)<q_a(1-)$ for $x\in(0,1)$ and $a\in\{0,1\}$. 
Of course, one may use any finite partition $0=x_0<x_1<\dotsb<x_{n-1}<x_n=1$ of the interval $(0,1)$ to obtain the corresponding piecewise quadratic bounds based on the inequalities $q_a(x_{i-1}+)<q_a(x)<q_a(x_i-)$ for $x\in(x_{i-1},x_i)$, for each $i=1,\dotsc,n$. 
We state this as another corollary, whose proof will be omitted due to its similarity to that of Corollary~\ref{cor:bounds}.
\end{remark}

\begin{cor}\label{cor:bounds2}
Let $0=x_0<x_1<\dotsc<x_{n-1}<x_n=1$. 
Then for $\are=\are_{T_1,T_2}$ and $q_a=q_{T_1,T_2;a}$ with $(T_1,T_2)\in\{(R,T),(T,S),(R,S)\}$ and $a\in\{0,1\}$, one has
\begin{equation*}
L_a\le L\le\are\le U\le U_a
\end{equation*}
on $(0,1)$, where
\begin{equation*}
L_a(x):=L_{T_1,T_2;a}(x):=
  \are(a)+\are'(a)(x-a)+q_a(x_{i-1}+)(x-a)^2
\end{equation*}
and
\begin{equation*}
U_a(x):=U_{T_1,T_2;a}(x):=
  \are(a)+\are'(a)(x-a)+q_a(x_i-)(x-a)^2
\end{equation*}
for $x\in(x_{i-1},x_i)$, and $L:=L_0\vee L_1$ and $U:=U_0\wedge U_1$.
\end{cor}

Corollary~\ref{cor:bounds2} is illustrated by Figure~\ref{fig:bounds3}; 
the bounds $L$ and $U$ are based on the partition $0=x_0<x_1<x_2=1$, where $x_1=(x_1)_{T_1,T_2}$ is chosen as the mean of the solutions to $L_0=L_1$ and $U_0=U_1$ (from Corollary~\ref{cor:bounds}), whose approximate values are given in Remark~\ref{re:num}. 
That is, $(x_1)_{R,T}\approx\frac12(0.7067+0.6573)\approx0.6820$, $(x_1)_{T,S}\approx\frac12(0.7969+0.7784)\approx0.7876$ and $(x_1)_{R,S}\approx\frac12(0.7916+0.7737)\approx0.7826$.

\begin{figure}[ht]
\begin{center}
\begin{tabular}{*{3}{c}}
\includegraphics[width=4.5cm]{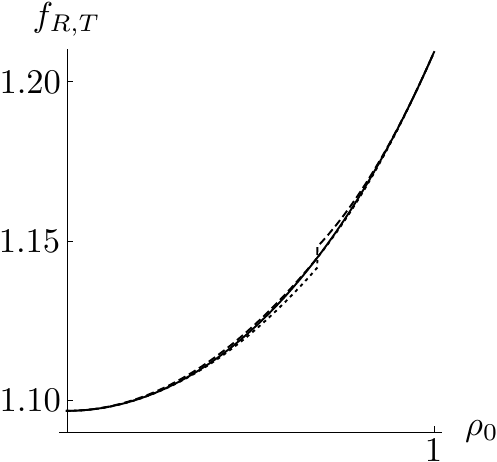}
&\includegraphics[width=4.5cm]{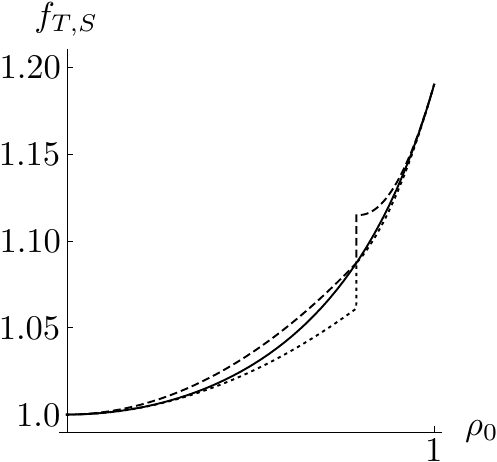}
&
\includegraphics[width=4.5cm]{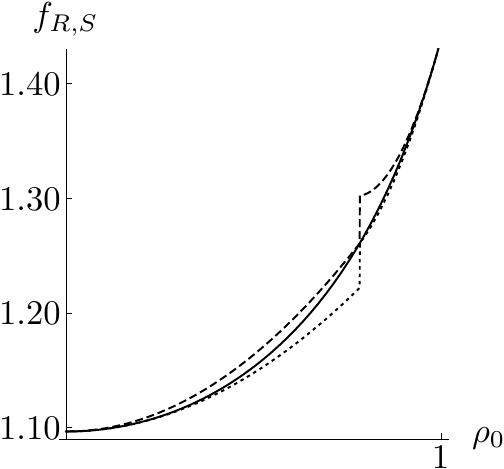}
\end{tabular}
\end{center}
\caption{Illustration of Corollary~\ref{cor:bounds2}, using partition $0=x_0<x_1<x_2=1$, where $(x_1)_{R,T}\approx0.682$, $(x_1)_{T,S}\approx0.788$, $(x_1)_{R,S}\approx0.783$; plots are of $f_{T_1,T_2}$ for $f=\are$ (solid), $f=L$ (dotted) and $f=U$ (dashed)}
\label{fig:bounds3}
\end{figure}

Note also that Corollary~\ref{cor:bounds} immediately implies even better quartic bounds on $\are_{R,S}$:
\begin{cor}\label{cor:quartic}
Let
\begin{equation*}
\tilde L_{R,S;a}:=L_{R,T;a}\cdot L_{T,S;a}
\quad\text{and}\quad
\tilde U_{R,S;a}:=U_{R,T;a}\cdot U_{T,S;a}
\end{equation*}
for $a\in\{0,1\}$, and also let
\begin{equation*}
\tilde L_{R,S}:=\tilde L_{R,S;0}\vee\tilde L_{R,S;1}
\quad\text{and}\quad
\tilde U_{R,S}:=\tilde U_{R,S;0}\wedge\tilde U_{R,S;1}.
\end{equation*}
Then
\begin{gather*}
L_{R,S;a}<\tilde L_{R,S;a}<\are_{R,S}<\tilde U_{R,S;a}<U_{R,S;a}
\intertext{for $a\in\{0,1\}$ and}
L_{R,S}<\tilde L_{R,S}<\are_{R,S}<\tilde U_{R,S}<U_{R,S}
\end{gather*}
on $(-1,0)\cup(0,1)$.
\end{cor}

\section{Proofs}\label{sec:proofs}
We first provide Berry-Ess\'een bounds for the distributions of the test statistics $R$, $S$, and $T$ and explicit expressions for the asymptotic mean and variance for each of these statistics. 
Once these facts are established, Theorem~\ref{thm:quad} will be proven with the aid of l'Hospital-type rules for determining the monotonicity pattern of a ratio.

\subsection{Berry-Ess\'een bounds and expressions for the asymptotic means and variances of R, S, and T}\label{sec:be}
Each of $R$, $S$, and $T$ shall be shown to satisfy \eqref{eq:ub}. 
For each of these statistics, it will be clear that $\V$ in \eqref{eq:ub} may be taken to be any open interval containing $\rho_0$ whose closure does not contain the points $-1$ or $1$. 
Further note that each of these three statistics is invariant to linear transformations of the form $X_i\mapsto aX_i+b$ and $Y_i\mapsto cY_i+d$ with $a>0$ and $c>0$. 
So, let us assume without loss of generality (w.l.o.g.) that $\mu_X=\mu_Y=0$ and $\si_X=\si_Y=1$. 
For convenience we allow the values $\rho=\pm1$; then the bivariate normal distribution is degenerate: $Y_i=\pm X_i$ a.s.

Based on the results of \cite[(4.9)]{nonlinear}, one has the following uniform Berry-Ess\'een bound on the distribution of $R$:
\begin{equation*}
\sup_{z\in\R}\Bigabs{\Bigprob[\rho]{\frac{R-\rho}{\si_R(\rho)/\sqrt{n}}\le z}-\Phi(z)}
  \le\frac{A}{\si_R^4(\rho)\sqrt{n}};
\end{equation*}
here and in what follows, $A$ stands for different positive absolute constants, and
\begin{equation}\label{eq:siR}
\si_R^2(\rho):= \E_\rho\bigl(X_1Y_1-\tfrac\rho2(X_1^2+Y_1^2)\bigr)^2=(1-\rho^2)^2;
\end{equation}
the last equality can be checked using the representation $Y_1=\rho X_1+\sqrt{1-\rho^2}Z_1$, where $X_1,Z_1\stackrel{iid}{\sim}N(0,1)$. 
Thus, \eqref{eq:ub} holds for $R$, with 
\begin{equation}\label{eq:muR}
\mu_R(\rho):=\rho.
\end{equation}

By \eqref{eq:T}, $T$ is a $U$-statistic with kernel $h_T$ of degree $m=2$. 
Further, $S$ (defined in \eqref{eq:S}) is a $V$-statistic with a kernel of degree $m=3$; 
Hoeffding \cite[Section 5c]{hoeff63} describes how any $V$-statistic can be expressed as a $U$-statistic of the same degree, so that $S$ is a $U$-statistic with a symmetric kernel $h_{S,n}$ of degree $m=3$. 
Namely, 
\begin{equation*}
S=\frac{1}{\binom n3}\sum_{1\le i<j<k\le n}h_{S,n}(V_i,V_j,V_k),
\end{equation*}
where
\begin{equation}\label{eq:hSn}
\begin{split}
h_{S,n}(V_i,V_j,V_k)
& :=\frac{n-2}{n+1}\,h_S(V_i,V_j,V_k)+\frac1{n+1}\,\bigl(h_T(V_i,V_j)+h_T(V_i,V_k)+h_T(V_j,V_k)\bigr),
\end{split}
\end{equation}
\begin{equation}\label{eq:hS}
h_S:=2(K_{ijk}+K_{ikj}+K_{jik}+K_{jki}+K_{kij}+K_{kji})-3
\end{equation}
and
\begin{equation*}
K_{ijk}:=\ind{X_j<X_i}\ind{Y_k<Y_i}.
\end{equation*}
It follows by Chen and Shao's result \cite[(3.4) in Theorem~3.1]{chen07} that for $m\in\{2,3\}$ and $n\ge m$
\begin{equation}\label{eq:ustat}
\sup_{z\in\R}\Bigabs{\Bigprob{\frac{U-\E U}{\si_1/\sqrt n}\le z}-\Phi(z)}
  \le\frac{A\,C^3}{\si_1^3\sqrt n},
\end{equation}
where $U=\binom nm^{-1}\sum_{1\le i_1<\dotsb<i_m\le n}h(V_{i_1},\dotsc,V_{i_m})$ is any $U$-statistic with a symmetric kernel $h$ such that $|h|\le C$ for some constant $C>0$, 
\begin{equation*}
\si_1^2:=m^2\var g(V_1)>0,
\end{equation*}
and
\begin{equation*}
g(V_1):=\E\bigl[h(V_1,\dotsc,V_m)|V_1\bigr].
\end{equation*}

Now consider $T$ as expressed in \eqref{eq:T}, and recall that $|h_T|=1$. 
One also has
\begin{equation}\label{eq:muT}
\mu_T(\rho):=\E_\rho T=\E_\rho h_T(V_1,V_2)
  =2\E_\rho\bigl(J_{12}+J_{21}\bigr)-1=4\E_\rho J_{12}-1
  =\tfrac 2\pi\asin\rho.
\end{equation}
In order to see this, note that $\E_\rho J_{12}=\P_\rho(X_1-X_2>0,Y_1-Y_2>0)=\P(Z_1>0,\rho Z_1+\sqrt{1-\rho^2}Z_2>0)$, where $Z_1$ and $Z_2$ are independent standard normal r.v.'s. 
By the circular symmetry of the distribution of $(Z_1,Z_2)$ on the plane, we see $\E_\rho J_{12}$ is simply the proportion of the length of the arc of the unit circle between the points $(0,1)$ and $(\sqrt{1-\rho^2},-\rho)$; 
that is,
\begin{equation*}
\E_\rho J_{12}=\tfrac{1}{2\pi}\bigl(\tfrac\pi2-\asin(-\rho)\bigr)=\tfrac14+\tfrac1{2\pi}\asin\rho,
\end{equation*}
whence \eqref{eq:muT}. 

One can use a similar geometric reasoning to obtain an expression for the asymptotic variance of $T$. 
Let
\begin{equation*}
g_T(V_1):=\E\bigl[h_T(V_1,V_2)|V_1\bigr]=2\E\bigl[J_{12}+J_{21}|V_1\bigr]-1,
\end{equation*}
so that
\begin{equation*}
\begin{split}
\si_T^2(\rho)
& :=4\var_\rho g_T(V_1)
  =16\bigl(\E_\rho J_{12}J_{13}+2\E_\rho J_{12}J_{31}+\E_\rho J_{21}J_{31}-4[\E_\rho J_{12}]^2\bigr).
\end{split}
\end{equation*}
Consider first $\E_\rho J_{12}J_{13}=\P(U_1>0,U_2>0,U_3>0,U_4>0)$, where the $U_i$'s are standard normal r.v.'s with
\begin{equation*}
\Si:=\cov\begin{bmatrix}U_1\\U_2\\U_3\\U_4\end{bmatrix}
=\begin{bmatrix}
  1&\frac12&\rho&\frac\rho2\\
  \frac12&1&\frac\rho2&\rho\\
  \rho&\frac\rho2&1&\frac12\\
  \frac\rho2&\rho&\frac12&1
\end{bmatrix}.
\end{equation*}
That is, $\E_\rho J_{12}J_{13}$ is the probability that the random point $\Si^{1/2}[Z_1,Z_2,Z_3,Z_4]\tr$ lies in the first orthant of 4-dimensional space, where the $Z_i$'s are independent standard normal r.v.'s; 
further, this is simply the ratio of the volume $V(\rho)$ of the spherical tetrahedron $A_1A_2A_3A_4$ to the volume $2\pi^2$ of the unit sphere $S_3:=\{x\in\R^4\colon\norm{x}=1\}$, where the vertices $A_1,A_2,A_3,A_4$ of the tetrahedron are the columns of $\Si^{-1/2}$ normalized to be unit vectors. 
One can use the classical result of Schl\"afli \cite{schlaf60} to obtain the volume of this spherical tetrahedron. 
But, in fact, this work has been indirectly done by David and Mallows in their derivation of the variance of $S$; the probabilities $\E_\rho J_{12}J_{13}$ and $\E_\rho J_{12}J_{31}$ correspond to correlation matrices (r) and (w), respectively, in Appendix 2 of \cite{david61}. 
Using the formulas there, and noting $\E_\rho J_{21}J_{31}=\E_\rho J_{12}J_{13}$ by the symmetry of the normal distribution, one sees
\begin{equation}\label{eq:siT}
\si_T^2(\rho)=\frac49-\frac{16}{\pi^2}\bigl(\asin\tfrac\rho2\bigr)^2,
\end{equation}
which is bounded away from 0 over any closed subinterval of $(-1,1)$, so, by \eqref{eq:ustat}, one has \eqref{eq:ub} for any $\th_0=\rho_0\in(-1,1)$.

We remark that Kendall's monograph \cite[Chapter 10]{ken48} contains derivations of \eqref{eq:muT} and \eqref{eq:siT}. 
Further, Plackett \cite{plack54} has obtained a more general method for calculating $\P(U_1>a_1,U_2>a_2,U_3>a_3,U_4>a_4)$ which reduces to the Schl\"afli method when the $a_i$ are all 0.

Directing attention to $S$, first note that $h_{S,n}$ is bounded (in fact, one can check that $\abs{h_{S,n}(V_1,V_2,V_3)}\in\{1,\frac{n-1}{n+1}\}$ a.s.). 
Using geometric reasoning similar to that used to compute $\E_\rho J_{12}$ (only now using the fact that $X_1-X_2$ and $Y_1-Y_3$ have a correlation of $\frac\rho2$), one finds
\begin{equation*}
\E_\rho K_{123}=\tfrac14+\tfrac1{2\pi}\asin\tfrac\rho2,
\end{equation*}
so that
\begin{equation*}
\mu_{S,n}(\rho):=\E_\rho S
  =\E_\rho h_{S,n}(V_1,V_2,V_3)
  =\tfrac{n-2}{n+1}\,\tfrac6\pi\,\asin\tfrac\rho2+\tfrac{3\mu_T(\rho)}{n+1};
\end{equation*}
accordingly, let
\begin{equation}\label{eq:muS}
\mu_S(\rho):=\lim_{n\to\infty}\mu_{S,n}(\rho)
  =\E_\rho h_S(V_1,V_2,V_3)
  =\tfrac6\pi\asin\tfrac\rho2
\end{equation}
and note that $\sqrt n(\mu_{S,n}-\mu_S)\conv_{n\to\infty}0$ uniformly for $\rho\in(-1,1)$.

Let next
\begin{align*}
\notag\begin{split}
g_{S,n}(V_1)
& :=\E\bigl[h_{S,n}(V_1,V_2,V_3)|V_1\bigr];
\end{split}
\\
\notag\begin{split}
g_S(V_1)
& :=\E\bigl[h_S(V_1,V_2,V_3)|V_1\bigr]
  =4\E\bigl[K_{123}+K_{213}+K_{231}|V_1\bigr]-3;
\end{split}
\\
\notag\begin{split}
\si_{S,n}^2(\rho)
& :=9\var_\rho g_{S,n}(V_1);
\end{split}
\\
\begin{split}
\si_S^2(\rho)
& :=9\var_\rho g_S(V_1)
\\
& =144\bigl(\E_\rho K_{123}K_{145}+2\E_\rho K_{213}K_{415}+4\E_\rho K_{123}K_{415}+2\E_\rho K_{213}K_{451}-9[\E_\rho K_{123}]^2\bigr),
\end{split}
\end{align*}
where $\si_{S,n}(\rho):=\sqrt{\si^2_{S,n}(\rho)}$ and $\si_S(\rho):=\sqrt{\si^2_S(\rho)}$), noting that $\E_\rho K_{231}K_{451}=\E_\rho K_{213}K_{415}$ and $\E_\rho K_{123}K_{451}$ $=\E_\rho K_{132}K_{415}=\E_\rho K_{123}K_{415}$ since the distributions of $\bigl((X_1,Y_1),\dotsc,(X_n,Y_n)\bigr)$ and $\bigl((Y_1,X_1),\dotsc,(Y_n,X_n)\bigr)$ are identical and permutation-invariant. 
It is clear that expressions for $\si_{S,n}^2$ and $\si_S^2$ may be derived in terms of the volumes of spherical tetrahedra via Schl\"afli's formula. 
For the sake of brevity, we omit these details and refer the reader to David and Mallows' derivation of $\var S$; 
note the probabilities $\E_\rho K_{123}K_{145}$, $\E_\rho K_{123}K_{415}$, $\E_\rho K_{213}K_{415}$, $\E_\rho K_{213}K_{451}$ correspond to the correlation matrices (c), (d), (e), and (f), respectively, found in Appendix 2 of \cite{david61}. 
Then one has
\begin{equation}\label{eq:siS}
\sigma^2_S(\rho)=1-\frac{324}{\pi^2}\,\bigl(\asin\tfrac \rho2\bigr)^2+\frac{72}{\pi^2}\,\bigl(I_1(\rho)+2I_2(\rho)+2I_3(\rho)+4I_4(\rho)\bigr),
\end{equation}
where
\begin{align*}
 I_1(x)&:=\int_0^x\frac{\asin\frac{u^3}{4(2-u^2)}}{\sqrt{4-u^2}}\,du,
&I_2(x)&:=\int_0^x\frac{\asin\frac{u}{2(3-u^2)}}{\sqrt{4-u^2}}\,du,\\
 I_3(x)&:=\int_0^x\frac{\asin\frac{u(4-u^2)}{2\sqrt2\sqrt{8-6u^2+u^4}}}{\sqrt{4-u^2}}\,du,
&I_4(x)&:=\int_0^x\frac{\asin\frac{u(4-u^2)}{2\sqrt{12-7u^2+u^4}}}{\sqrt{4-u^2}}\,du;
\end{align*}
an explicit expression of $\si_{S,n}^2$ is not of direct concern to us and so is omitted (though could also be obtained from \cite{david61}). 
Note the integrals $I_1,\dotsc,I_4$ are expressed differently than the corresponding ones found in \cite{david61}, though a simple change of variables shows their equivalence 
\big(the integrals equivalent to $I_1;I_2;I_3;I_4$ are found in Appendix 2 of \cite{david61} in the expressions corresponding to the correlation matrices labeled there by (f); (c); (f); (d) and (e), respectively.\big) 

Now, $\si_{S,n}\conv_{n\to\infty}\si_S$ uniformly over all $\rho\in[-1,1]$ (since, by \eqref{eq:hSn}, $h_{S,n}-h_S=O(1/n)$). 
It will be pointed out in the last paragraph of part (TS0) of the proof of Theorem~\ref{thm:quad} that $\si_S^2>0$ for $\rho\in(-1,1)$. 
It is also clear from \eqref{eq:siS} that $\si_S^2$ is a continuous function of $\rho$, so that the minimum of $\si_S$ over any closed subinterval of $(-1,1)$ is strictly positive. Thus, $\inf_{\rho\in\V}\si_{S,n}(\rho)>0$ for all large enough $n$, where $\V$ is as introduced in the beginning of Section~\ref{sec:be}. 
Referring now to \eqref{eq:ustat} (and replacing there $U$ with $S$, $\E U$ with $\mu_{S,n}$ and $\si_1$ with $\si_{S,n}$), one finds that
\begin{equation*}
\sup_{z\in\R}\Bigabs{\Bigprob[\rho]{\frac{S-\mu_S(\rho)}{\si_S(\rho)/\sqrt n}\le z}-\Phi(z)}\le\frac{A}{\si_{S,n}^3(\rho)\sqrt n}+\bigabs{\Phi(z^\ast)-\Phi(\tfrac{\si_S(\rho)}{\si_{S,n}(\rho)}\,z)}+\bigabs{\Phi(\tfrac{\si_S(\rho)}{\si_{S,n}(\rho)}\,z)-\Phi(z)},
\end{equation*}
where $z^\ast=\frac{\si_S(\rho)}{\si_{S,n}(\rho)}\big(z+\frac{\mu_S(\rho)-\mu_{S,n}(\rho)}{\si_S(\rho)/\sqrt n}\big)$; in turn, the last two terms in the above inequality vanish uniformly over $z\in\R$ and $\rho\in\V$ as $n$ tends to $\infty$ (using well-known properties of the function $\Phi$ and the previously noted facts that $\sqrt n(\mu_S-\mu_{S,n})\to0$ and $\si_{S,n}/\si_S\to1$ uniformly on $\V$), so that $S$ satisfies \eqref{eq:ub}. 

The next result will be used in the proofs of the statements (TS0)~--~(RS1) in Theorem~\ref{thm:quad}:
\begin{lem}\label{lem:siS(1)=0}
One has $\si_S^2(1-)=0$.
\end{lem}
\begin{proof}
W.l.o.g.,\ $Y_i=\rho X_i+\sqrt{1-\rho^2}Z_i$ for all $i$, where the $Z_i$'s are i.i.d. $N(0,1)$ r.v.'s independent of the $X_i$'s. 
Further note that $\si_{S,n}^2(\rho)$ differs only by a positive constant factor from $\var_\rho\operatorname{proj}_{\mathcal L}S$, where $\mathcal L$ is the space of all linear statistics.  
Also, for $\rho=1$, one has $S=1$ a.s.\ and hence $\var_\rho\operatorname{proj}_{\mathcal L}S\le\var_\rho S=0$, so $\si^2_{S,n}(1)=0$ for all $n$. 
Now, letting $n\to\infty$, one has $\si^2_S(1)=0$, since $h_{S,n}-h_S=O(1/n)$.

Next, $\frac19\si^2_S(\rho)=\var_\rho g_S(V_1)=\E_\rho h_S(W_1,W_2,W_3)h_S(W_1,W_4,W_5)-\E^2_\rho h_S(W_1,W_2,W_3)$, with $W_i:= W_i(\rho):=(X_i,\rho X_i +\sqrt{1-\rho^2}Z_i)$. 
Next, $h_S(W_1,W_2,W_3)$ and $h_S(W_1,W_4,W_5)$ are continuous in $\rho$ on the complement of the union of all events of the form $\{X_i=X_j\}$ for $i\ne j$. 
The latter union has zero probability. 
So, by dominated convergence, $\si^2_S(\rho)\to\si^2_S(1)=0$ as $\rho\uparrow 1$. 
\end{proof}

While the result of this last lemma should not be surprising, it should be noted that trying to assert $\si_S^2(1-)=0$ using only the expression \eqref{eq:siS} is a more difficult task.

\subsection{Proofs of monotonicity}\label{sec:proof.mon}
As in \cite{pin01,pin02mon,pin06}, let $-\infty\le a<b\le\infty$, and suppose that $f$ and $g$ are differentiable functions on $(a,b)$. 
Let $r:=\frac fg$ and $\rho:=\frac{f'}{g'}$; 
from hereon, the symbol $\rho$ should not be considered the correlation of a bivariate normal population, which latter will be denoted by $x$. 
Assume that either $g<0$ or $g>0$ on $(a,b)$, and also that $g'<0$ or $g'>0$ on $(a,b)$. 
For an arbitrary function $h$ defined on $(a,b)$, adopt the notation ``$h$ $\inc$'' to mean $h$ is (strictly) increasing on $(a,b)$ and similarly let ``$h$ $\dec$'' mean $h$ is decreasing on $(a,b)$; 
the juxtaposition of these arrows shall have the obvious meaning, e.g.~``$h$ $\inc\,\dec$'' means that there exists some $c\in(a,b)$ such that $h$ $\inc$ on $(a,c)$ and $h$ $\dec$ on $(c,b)$. 
Further, let the notation ``$h$ is $+-$'' mean that there exists $c\in(a,b)$ such that $h>0$ on $(a,c)$ and $h<0$ on $(c,b)$; 
similar meaning will be given to other such strings composed of alternating ``$+$'' and ``$-$'' symbols.

\begin{thmA}\label{thm:special}\ \\
Suppose that either $f(a+)=g(a+)=0$ or $f(b-)=g(b-)=0$.
\begin{enumerate}[(i)]
\item If $\rho$ $\inc$ on $(a,b)$, then $r'>0$ on $(a,b)$ and hence $r$ $\inc$ on $(a,b)$;
\item If $\rho$ $\dec$ on $(a,b)$, then $r'<0$ on $(a,b)$ and hence $r$ $\dec$ on $(a,b)$.
\end{enumerate}
\end{thmA}

\begin{thmB}\label{thm:general}\ 
\begin{enumerate}[(i)]
\item If $\rho$ $\inc$ and $gg'>0$ on $(a,b)$, then $r$ $\dec$, $r$ $\inc$ or $r$ $\dec\,\inc$ on $(a,b)$;
\item If $\rho$ $\inc$ and $gg'<0$ on $(a,b)$, then $r$ $\dec$, $r$ $\inc$ or $r$ $\inc\,\dec$ on $(a,b)$;
\item If $\rho$ $\dec$ and $gg'>0$ on $(a,b)$, then $r$ $\dec$, $r$ $\inc$ or $r$ $\inc\,\dec$ on $(a,b)$;
\item If $\rho$ $\dec$ and $gg'<0$ on $(a,b)$, then $r$ $\dec$, $r$ $\inc$ or $r$ $\dec\,\inc$ on $(a,b)$.
\end{enumerate}
\end{thmB}

\begin{thmC}\label{thm:refined}\ \\
Let $\tilde\rho:=g^2\,\frac{r'}{|g'|}=\sign(g')(\rho g-f)$. 
\begin{enumerate}[(i)]
\item If $\rho$ $\inc$ and $gg'>0$ on $(a,b)$, then $\tilde\rho$ $\inc$;
\item If $\rho$ $\inc$ and $gg'<0$ on $(a,b)$, then $\tilde\rho$ $\dec$;
\item If $\rho$ $\dec$ and $gg'>0$ on $(a,b)$, then $\tilde\rho$ $\dec$;
\item If $\rho$ $\dec$ and $gg'<0$ on $(a,b)$, then $\tilde\rho$ $\inc$.
\end{enumerate}
In addition, $\sign(\tilde\rho)=\sign(r')$, so that the monotonicity pattern of $r$ may be determined by the monotonicity of $\tilde\rho$ and knowledge of the signs of $\tilde\rho(a+)$ and/or $\tilde\rho(b-)$.
\end{thmC}

E.g.\ suppose it can be established that $\rho$ $\inc$ and $gg'>0$ on $(a,b)$; 
if one also knows that $r(a+)=-\infty$ then the general rules imply $r$ $\inc$. 
Alternatively, $\rho$ $\inc$ and $gg'>0$ imply $\tilde\rho$ $\inc$; 
if it can be established that $\tilde\rho(a+)\ge0$, then $\tilde\rho>0$ on $(a,b)$ and hence $r$ $\inc$ on $(a,b)$. 
We shall make frequent use of these rules throughout the proof of Theorem~\ref{thm:quad}. 
The special-case rules are proved in \cite[Proposition 1.1]{pin02mon}, and a proof of the general rules is found in \cite[Proposition 1.9]{pin01}. 
A proof of the refined general rules, along with several other variants of these monotonicity rules, is found in \cite[Lemma 2.1]{pin06}. 
Note that Anderson et al.~\cite[Lemma 2.2]{and93} proved a variant of the special-case rules, wherein the function $\frac{f(x)-f(a)}{g(x)-g(a)}$ $\inc$ (or $\dec$) whenever $\rho$ $\inc$ (or $\dec$).

That \eqref{eq:are} may be used to express any of the three pairwise $\are$'s has been justified by the work of the previous section. 
The proofs of the six statements (RT0)~--~(RS1) in Theorem~\ref{thm:quad} will follow the same general method. 
Fix an arbitrary $a\in[0,1]$, and let 
\begin{equation}\label{eq:b,c}
b:=b(a):=\are(a)\quad\text{and}\quad c:=c(a):=\are'(a). 
\end{equation}
Then
\begin{equation*}
q_a(x)=\frac{\are(x)-b-c(x-a)}{(x-a)^2}=\frac{f(x)-bg(x)-c(x-a)g(x)}{(x-a)^2g(x)}
\end{equation*}
when $f$ and $g$ are functions chosen so that $\are=\frac fg$. 
Accordingly, let
\begin{equation}
\label{eq:scheme}
\begin{split}
&
f_0(x):= f(x)-bg(x)-c(x-a)g(x),\quad
g_0(x):= (x-a)^2g(x),\quad
r_0(x):=\frac{f_0(x)}{g_0(x)}=q_a(x),
\\
&
f_i:= a_if_{i-1}',\quad
g_i:= a_ig_{i-1}',\quad
r_i:=\frac{f_i}{g_i},\quad
\rho_{i-1}:=\frac{f_{i-1}'}{g_{i-1}'}=r_i,
\text{ and }
\tilde\rho_i=\sign(g_{i+1})\bigl(r_{i+1}g_i-f_i\bigr)
\end{split}
\end{equation}
where the $a_i$ are positive on $(0,1)$. 
There is some freedom in choosing the functions $a_i$, though the goal is to ensure that, for some natural number $n\ge1$, the ratio $r_n$ is an algebraic function. 
In our case it will turn out that $r_n$ is actually an algebraic function independent of the value of $a$. 
As $r_n$ is algebraic, the problem of determinining its monotonicity pattern on an interval is completely algorithmic (cf.~\cite{tarski48,collins98}); 
here, we use the Mathematica \texttt{Reduce} command to deduce the monotonicity of $r_n=\rho_{n-1}$. 
The specific choices of $f$, $g$ and the $a_i$ are given in Lemmas~\ref{lem:q_RT}~--~\ref{lem:q_RS} below. 
One may refer to this first phase of the proof as the ``reduction'' phase.

Once the monotonicity of $r_n=\rho_{n-1}$ is established, the second and final stage of the proof is to ``work backwards'' by using the various l'Hospital-type rules stated above to deduce the monotonicity patterns of $r_{n-1}=\rho_{n-2}$, $r_{n-2}=\rho_{n-3},\dotsc$, $r_1=\rho_0$, $r_0=q_a$. 
Throughout the proof, all functions shall be assumed to be defined on $(0,1)$ unless otherwise stated. 

As most of the functions being treated are rather unwieldy, all calculations are performed with the Mathematica (v.~5.2 or later) software; 
detailed output from the notebooks has been reproduced as the appendices, and the actual notebooks will be made available upon request. 
Each of the appendices~RT, TS, and RS follows the same general format: the first section (labeled RTr, TSr, or RSr -- where ``r'' stands for ``reduction (phase)'') is dedicated to proving the corresponding one of the Lemmas \ref{lem:q_RT}--\ref{lem:q_RS} below (i.e., the ``reduction'' stage of the proofs), the second section (RT0, TS0, or RS0) provides numerical support for proving the monotonicity of $q_0$, and the third section (RT1, TS1, or RS1) provides support for proving the monotonicity of $q_1$. 

We prove $q_a$ is increasing only for $a\in\{0,1\}$; 
the following three lemmas could perhaps be used as starting points for the ``working backwards'' phase for other choices of $a\in(0,1)$ to get even more quadratic bounds on the $\are$'s (cf.\ Corollary~\ref{cor:bounds}). 
It is of course desirable to demonstrate that $q_a\,\inc$ for arbitrary $a\in[0,1]$ (should this be true), though a proof of such a statement has yet to be found; 
for any given $a\in(0,1)$, this second phase of the proof is restricted only by computational capacities, since, as mentioned above, the expression for $r_n$ is eventually algebraic. 
We remark also that this method could conceivably be adapted (by using an appropriate variant of the definition of $q_a$) to finding quadratic bounds on $\are_{T,R}=1/\are_{R,T}$, $\are_{S,T}=1/\are_{T,S}$, and $\are_{S,R}=1/\are_{R,S}$, or possibly finding approximating polynomials of degree greater than 2. 

\begin{lem}\label{lem:q_RT}
Let $a\in[0,1]$ be arbitrary, and let
\begin{equation*}
f(x):=\pi^2-36\bigl(\asin\tfrac x2\bigr)^2\quad\text{and}\quad
g(x):=9\bigl(1-x^2\bigr)
\end{equation*}
for $x\in(0,1)$. Then on the interval $(0,1)$, one has $\are_{R,T}=\frac fg$, $r_4$ $\inc$, $f_4<0$ and $g_4<0$, where
\begin{gather*}
a_1(x):=\sqrt{4-x^2},\quad
a_2(x):=\frac{\sqrt{4-x^2}}{2-x^2},\quad
a_3(x):=\frac{(2-x^2)^2}{50-29x^2+9x^4},\quad
a_4(x):=\frac{(50-29x^2+9x^4)^2}{2-x^2},
\end{gather*}
and $f_i$, $g_i$, $r_i$ are as defined in \eqref{eq:scheme}.
\end{lem}

\begin{proof}
From \eqref{eq:muR} and \eqref{eq:muT}, $\mu_R'(x)=1$ and $\mu_T'(x)=\frac2\pi(1-x^2)^{-1/2}$; then $\are_{R,T}=\frac fg$ upon recalling \eqref{eq:are}, \eqref{eq:siR}, and \eqref{eq:siT}. 
Visual inspection shows that $a_i>0$ on $(0,1)$ for $i=1,\dotsc,4$. 
It is easily verified that $f_4<0$, $g_4<0$ and $r_4'>0$ on $(0,1)$; see Appendix~RTr for explicit expressions of these and the intermediate functions.
\end{proof}

\begin{lem}\label{lem:q_TS}
Let $a\in[0,1]$ be arbitrary, and let
\begin{equation}\label{eq:f,g}
f(x):=\si_S^2(x)\quad\text{and}\quad
g(x):=\frac{4(1-x^2)(\pi^2-36(\asin\frac x2)^2)}{\pi^2(4-x^2)},
\end{equation}
where $\si_S^2$ is given in \eqref{eq:siS}. Then on the interval $(0,1)$, one has $\are_{T,S}=\frac fg$, $r_{10}$ $\inc$, $f_{10}>0$ and $g_{10}>0$, where
\begin{gather*}
a_1(x):=\sqrt{4-x^2},\quad
a_2(x):=\frac{(4-x^2)^{5/2}}{2+x^2},\quad
a_3(x):=\frac{(2+x^2)^2}{(4-x^2)(38-17x^2-3x^4)},
\end{gather*}
and $a_4,\dotsc,a_{10}$ are functions rational in $x$ and $\sqrt{4-x^2}$, which are positive and continuous on $(0,1)$, with $f_i$, $g_i$, and $r_i$ as defined in \eqref{eq:scheme}.
\end{lem}

\begin{proof}
The proof is found in Appendix~TSr.
\end{proof}

\begin{lem}\label{lem:q_RS}
Let $a\in[0,1]$ be arbitrary, and let
\begin{equation*}
f(x):=\si_S^2(x)\quad\text{and}\quad
g(x):=\frac{36(1-x^2)^2}{\pi^2(4-x^2)},
\end{equation*}
where $\si_S^2$ is given in \eqref{eq:siS}. Then on the interval $(0,1)$, one has $\are_{R,S}=\frac fg$, $r_5$ $\inc$, $f_5>0$, and $g_5>0$, where
\begin{gather*}
a_1(x):=\sqrt{4-x^2},\quad
a_2(x):=(4-x^2)^{5/2},\quad
a_3(x):=\frac{1}{x(41-20x^2+3x^4)}, 
\end{gather*}
and $a_4,a_5$ are rational functions, which are positive and continuous on $(0,1)$, with $f_i$, $g_i$, and $r_i$ as defined in \eqref{eq:scheme}.
\end{lem}

\begin{proof}
The proof is given in Appendix~RSr.
\end{proof}

Before proving Theorem~\ref{thm:quad}, recall the implications of \eqref{eq:scheme}. 
If on some open subinterval of $(0,1)$ one has $f_i>0$ (or $f_i<0$), then on this subinterval $f_{i-1}$ $\inc$ (or $f_{i-1}$ $\dec$), and similarly for the $g_i$'s. 
If $g_i$ has $k$ roots in $(0,1)$, these shall be denoted by $x_{i,j}$, $j=1,\dotsc,k$, with the assumption that $x_{i,1}<\dotsb<x_{i,k}$; 
if $g_i$ has only a single root in $(0,1)$, it will simply be denoted by $x_i$. 
Similarly, the roots of $f_i$ whenever they exist will be denoted by $y_{i,1},y_{i,2},\dotsc$ (or simply $y_i$ if $f_i$ has a single root), and if ever $r_i'$ is shown to have a root in $(0,1)$ (there will only be at most one root in what follows), this root will be denoted by $z_i$. 
Numerical approximations of any of these roots are not of direct concern to us, but rather their positions relative to other roots. 
Such information is easily obtained from evaluation of the respective functions at specific points; 
for instance, if at some step we deduce that $f_1$ and $g_1$ are both $+-$, with $f_1(0.5)>0>g_1(0.5)$, then it is inferred that $x_1<0.5<y_1$ (and further, that $r_1(x_1-)=\frac{f_1}{g_1}(x_1-)=\infty$ and $r_1(x_1+)=\frac{f_1}{g_1}(x_1+)=-\infty$).

\begin{proof}[Proof of Theorem~\ref{thm:quad}, (RT0)]
See Appendix~RT0 for more details of the following arguments. 
Adopt the notation of Lemma~\ref{lem:q_RT}, with $a=0$, so that, in accordance with \eqref{eq:b,c}, $b=\are_{R,T}(0)=\frac{\pi^2}9$ and $c=\are_{R,T}'(0)=0$. 
Noting that $f_3(0+)=g_3(0+)=0$, one has $f_3<0$, $g_3<0$ (since, by Lemma~\ref{lem:q_RT}, $f_4<0$ and $g_4<0$), and also, by the special-case rules, $\rho_2=r_3\,\inc$ (since, by Lemma~\ref{lem:q_RT}, $\rho_3=r_4\,\inc$).

Next, $g_2\,\dec$ (as $g_3<0$) and $g_2(0+)>0>g_2(1-)$ imply $g_2>0$ on $(0,x_2)$ and $g_2<0$ on $(x_2,1)$; 
similarly, $f_2\,\dec$ and $f_2(0+)>0>f_2(1-)$ imply $f_2>0$ on $(0,y_2)$ and $f_2<0$ on $(y_2,1)$. 
Verifying that $g_2(0.41)<0<f_2(0.41)$, one has $x_2<0.41<y_2$, further implying $r_2(x_2-)=\infty$ and $r_2(x_2+)=-\infty$. 
Noting the sign of $g_2g_2'$ (which is the sign of $g_2g_3$) on each of $(0,x_2)$ and $(x_2,1)$, the general rules imply $\rho_1=r_2\,\inc$ on each of these two intervals.

Next, $g_1\,\inc\,\dec$ on $(0,1)$ (as $g_2$ is $+-$) and $g_1(0+)=0>g_1(1-)$ imply the existence of a single root $x_1$, with $x_2<x_1$; 
similarly, $f_1\,\inc\,\dec$ and $f_1(0+)=0>f_1(1-)$ imply the existence of a single root $y_1$, with $y_2<y_1$. 
The special-case rules imply $r_1\,\inc$ on $(0,x_2)$ (as $f_1(0+)=g_1(0+)=0$). 
Further, $g_1(0.71)<0<f_1(0.71)$ implies $x_1<y_1$, which in turn shows $r_1(x_1-)=\infty$ and $r_1(x_1+)=-\infty$; 
noting the sign of $g_1g_1'$ on each of the intervals $(x_2,x_1)$ and $(x_1,1)$, the general rules imply $r_1\,\inc$ on these two intervals. 
The continuity of $r_1$ at $x_2$ implies $\rho_0=r_1\,\inc$ on $(0,x_1)$ and $(x_1,1)$.

Finally, $f_0(0+)=g_0(0+)=f_0(1-)=g_0(1-)=0$ imply both $g_0>0$ on $(0,1)$ (since $g_1$ is $+-$ and hence $g_0\,\inc\,\dec$ on $(0,1)$) and $r_0\,\inc$ on each of the intervals $(0,x_1)$ and $(x_1,1)$ (by the special-case rules); 
the continuity of $r_0$ at $x_1$ implies $q_{R,T;0}=r_0\,\inc$ on $(0,1)$. 
Further, the l'Hospital rule for limits implies $r_0(0+)=r_2(0+)$ and $r_0(1-)=r_1(1-)$.
\end{proof}

\begin{proof}[Proof of Theorem~\ref{thm:quad}, (RT1)]
See Appendix~RT1 for more details of the following arguments. 
Adopt the notation of Lemma~\ref{lem:q_RT}, with $a=1$, so that
\begin{gather*}
b=\are_{R,T}(1-)=\tfrac{f}{g}(1-)=\tfrac{f'(1-)}{g'(1-)}=\tfrac{2\pi}{3\sqrt3}
\intertext{and}
c=\are_{R,T}'(1-)=\tfrac{f'g-fg'}{g^2}(1-)=\tfrac{f''g-fg''}{2gg'}(1-)=\tfrac{f''(1-)g'(1-)-f'(1-)g''(1-)}{2g'(1-)^2},
\end{gather*}
which follows by repeated application of the l'Hospital rule for limits after noting $f(1-)=g(1-)=0$. 

Next, $g_3(0+)>0>g_3(1-)$ and $f_3(0+)>0>f_3(1-)$ along with $g_3\,\dec$ and $f_3\,\dec$ (since $f_4<0$ and $g_4<0$ by Lemma~\ref{lem:q_RT}) shows that $g_3$ and $f_3$ each have a single root $x_3$ and $y_3$, respectively. 
Also, $g_3(0.6)<0<f_3(0.6)$ shows $x_3<y_3$ and hence $r_3(x_3-)=\infty$ and $r_3(x_3+)=-\infty$. 
Noting the sign of $g_3g_3'$ on each of the intervals $(0,x_3)$ and $(x_3,1)$, the general rules imply $\rho_2=r_3\,\inc$ on these two intervals.

Next, $g_2\,\inc\,\dec$ (as $g_3$ is $+-$) and $g_2(0+)=g_2(1-)=0$ imply $g_2>0$, whereas $f_2\,\inc\,\dec$ and $f_2(0+)<0=f_2(1-)$ imply $f_2$ has a single root $y_2$. 
The special-case rules imply $r_2\,\inc$ on $(x_3,1)$; 
as $\rho_2\,\inc$ and $g_2g_2'>0$ on $(0,x_3)$ and $\tilde\rho_2(0+)>0$, the refined general rules imply $\tilde\rho_2>0$ and hence $r_2\,\inc$ on $(0,x_3)$. 
Noting that $r_2$ is continuous at $x_3$, one has $\rho_1=r_2\,\inc$ on $(0,1)$.

Next, $g_1\,\inc$ and $f_1(1-)=g_1(1-)=0$ imply both $g_1<0$ and $\rho_0=r_1\,\inc$ on $(0,1)$; 
similarly, $g_0\,\dec$ and $f_0(1-)=g_0(1-)=0$ imply $g_0>0$ and $q_{R,T;1}=r_0\,\inc$ on $(0,1)$. 
Lastly, $r_0(0+)=\frac{f_0(0+)}{g_0(0+)}$ and also $r_0(1-)=r_3(1-)$, which follows by the l'Hospital rule for limits.
\end{proof}

\begin{proof}[Proof of Theorem~\ref{thm:quad}, (TS0)]
See Appendix~TS0 for more details of the following arguments. 
Adopt the notation of Lemma~\ref{lem:q_TS}, with $a=0$, so that $b=\are_{T,S}(0)=1$ and $c=\are_{T,S}'(0)=0$. 
Now, $g_9\,\inc$, $f_9\,\inc$, and $f_9(0+)=g_9(0+)=0$ imply $f_9>0$, $g_9>0$, and $\rho_8=r_9\,\inc$ (using the results of Lemma~\ref{lem:q_TS} and the special-case rules) on $(0,1)$. 
Also, $g_8(1-)<0$, $f_8(0+)>0$, and $\tilde\rho_8(0+)<0$ imply $g_8<0$, $f_8>0$, and $\rho_7=r_8\,\dec$ (by the refined general rules) on $(0,1)$. 
Further, $f_7(0+)=g_7(0+)=0$ imply $f_7>0$, $g_7<0$, and $\rho_6=r_7\,\dec$ (again by the special-case rules) on $(0,1)$.

Next, $g_6\,\dec$ and $g_6(0+)>0>g_6(1-)$ imply the existence of a single root $x_6$; 
$f_6\,\inc$ and $f_6(0+)>0$ imply $f_6>0$ on $(0,1)$. 
The refined general rules imply $r_6\,\inc$ on $(0,x_6)$ (as $\tilde\rho_6(0+)>0$), and also that $\tilde\rho_6\,\dec$ on $(x_6,1)$. 
As $x_6<0.75$ (since $g_6(0.75)<0$), note that $\tilde\rho_6(x_6+)>\tilde\rho_6(0.75)>0>\tilde\rho_6(1-)$ implies $r_6\,\inc\,\dec$ on $(x_6,1)$. 
That is, $r_6'$ has a single root $z_6$, and hence we have $\rho_5=r_6\,\inc$ on each of $(0,x_6)$ and $(x_6,z_6)$ and $\dec$ on $(z_6,1)$.

Next, $g_5(0+)>0$ and $g_5(1-)>0$ (along with $g_5\,\inc\,\dec$) imply $g_5>0$ on $(0,1)$; 
also, $f_5(0+)>0$ implies $f_5>0$ on $(0,1)$. 
As $x_6>0.5$ (since $g_6(0.5)>0$) and $\tilde\rho_5\,\inc$ on $(0,x_6)$ (by the refined general rules), one has $\tilde\rho_5(0+)<0<\tilde\rho_5(0.5)<\tilde\rho_5(x_6+)$; 
that is, $r_5\,\dec\,\inc$ on $(0,x_6)$, or $r_5'$ has a single root $z_5$ (with $z_5<x_6$). 
Recall that $f_5$, $f_5'$ and $g_5$ are all positive on $(0,1)$, and also $g_5'<0$ on $(x_6,1)$. 
Then $r_5'=\frac{f_5'g_5-f_5g_5'}{g_5^2}>0$ and hence $r_5\,\inc$ on $(x_6,1)$. 
$\bigl($Let us remark at this point that the l'Hospital-type rules could, in principle, be used to establish the monotonicity of $r_5$ on each of $(x_6,z_6)$ and $(z_6,1)$; however, this would necessitate proving that $\tilde\rho_5(z_6)>0$, a task which requires more work than simply requesting the Mathematica program to evaluate the function at the \emph{approximation} of the root $z_6$.$\bigr)$ 
As $r_5$ is continuous on $(0,1)$, we have $\rho_4=r_5\,\dec$ on $(0,z_5)$ and $\inc$ on $(z_5,1)$. 

Next, $g_4(0+)=-\infty<0<g_4(1-)$ and $f_4(0+)=-\infty<0<f_4(1-)$ imply the existence of roots $x_4$ and $y_4$ (as $g_5>0$ and $f_5>0$). 
As $g_4(0.3)<0<r_5'(0.3)$, we see that $x_4>0.3>z_5$; 
the refined general rules imply $\tilde\rho_4\,\inc$ on $(0,z_5)$, and so, $\tilde\rho_4(0+)=0$ implies $r_4\,\inc$ on $(0,z_5)$. 
Also, $g_4(0.4)>0>f_4(0.4)$ implies $x_4<0.4<y_4$, so that $r_4(x_4-)=\infty$ and $r_4(x_4+)=-\infty$. 
The general rules then imply $r_4\,\inc$ on each of $(z_5,x_4)$ and $(x_4,1)$. 
Further, the continuity of $r_4$ at $z_5$ implies $\rho_3=r_4\,\inc$ on both $(0,x_4)$ and $(x_4,1)$.

Next, $g_3\,\dec\,\inc$ and $g_3(0+)=0<g_3(1-)$ imply the existence of a single root $x_3$; at that, $x_3>x_4$; 
similarly, $f_3(0+)=0<f_3(1-)$ implies the existence of $y_3$. 
The special-case rules imply $r_3\,\inc$ on $(0,x_4)$; 
$g_3(0.64)>0>f_3(0.64)$ implies $x_3<0.64<y_3$, or $r_3(x_3-)=\infty$ and $r_3(x_3+)=-\infty$, so that the general rules show that $r_3\,\inc$ on $(x_4,x_3)$ and $(x_3,1)$. 
As $r_3$ is continuous at $x_4$, one has $\rho_2=r_3\,\inc$ on $(0,x_3)$ and $(x_3,1)$.

Next, $g_2\,\dec\,\inc$, along with $g_2(0+)>0>g_2(0.5)$ and $g_2(1-)>0$, implies the existence of two roots $x_{2,1}$ and $x_{2,2}$; 
similarly, $f_2(0+)>0>f_2(0.5)$ and $f_2(1-)>0$ shows $f_2$ has two roots $y_{2,1},y_{2,2}$. 
Noting that $g_2(0.35)<0<f_2(0.35)$ and also $g_2(0.86)>0>f_2(0.86)$, we have $x_{2,1}<0.35<y_{2,1}<0.5<x_{2,2}<0.86<y_{2,2}$, whence $r_2(x_{2,1}-)=r_2(x_{2,2}-)=\infty$ and $r_2(x_{2,1}+)=r_2(x_{2,2}+)=-\infty$; 
the general rules then imply that $r_2\,\inc$ on each of $(0,x_{2,1})$, $(x_{2,1},x_3)$, $(x_3,x_{2,2})$ and $(x_{2,2},1)$. 
The continuity of $r_2$ at $x_3$ implies $\rho_1=r_2\,\inc$ on $(0,x_{2,1})$, $(x_{2,1},x_{2,2})$ and $(x_{2,2},1)$.

Next, $f_1(0+)=g_1(0+)=f_1(1-)=g_1(1-)=0$ (together with $f_2$ and $g_2$ both $+-+$) implies the existence of roots $x_1$ and $y_1$. 
That $r_1\,\inc$ on $(0,x_{2,1})$ and $(x_{2,2},1)$ is implied by the special-case rules; 
that $r_1\,\inc$ on $(x_{2,1},x_1)$ and $(x_1,x_{2,2})$ is implied by the general rules upon noting that $g_1(0.62)<0<f_1(0.62)$ (and hence $x_1<y_1$, or $r_1(x_1-)=\infty$ and $r_1(x_1+)=-\infty$). 
The continuity of $r_1$ at $x_{2,1}$ and $x_{2,2}$ implies $\rho_0=r_1\,\inc$ on $(0,x_1)$ and $(x_1,1)$.

Lastly, $f_0(0+)=g_0(0+)=f_0(1-)=g_0(1-)=0$ shows $g_0>0$ on $(0,1)$ and also, by the special-case rules, $r_0\,\inc$ on $(0,x_1)$ and $(x_1,1)$. 
The continuity of $r_0$ at $x_1$ shows $q_{T,S;0}=r_0\,\inc$ on $(0,1)$. 
Further, the l'Hospital rule for limits yields $r_0(0+)=r_2(0+)$ and $r_0(1-)=r_2(1-)$.

As promised in the remarks preceding Lemma~\ref{lem:siS(1)=0}, we show that $\si_S>0$ on $(0,1)$ (and hence on $(-1,0)$ as $\si_S$ is even). 
Note $f_0>0$ (as $f_0\,\inc\,\dec$ and $f_0(0+)=f_0(1-)=0$); 
by \eqref{eq:f,g} and \eqref{eq:scheme}, and recalling that $b=1$ and $c=0$, one has $f_0=\si_S^2-g$, so that $\si_S^2>g$ on $(0,1)$. 
As $x^2g(x)=g_0(x)>0$, it follows that $\si_S^2>0$. Further note that there is no circular reasoning here; the above proof stands on its own, regardless of any probabilistic interpretation we give to the functions $f$ or $g$.
\end{proof}

\begin{proof}[Proof of Theorem~\ref{thm:quad}, (TS1)]
See Appendix~TS1 for more details of the following arguments. 
Adopt the notation of Lemma~\ref{lem:q_TS}, with $a=1$, so that $f(1-)=g(1-)=f'(1-)=g'(1-)=0$, and repeated application of the l'Hospital rule for limits imply
\begin{equation}\label{eq:lHosp.b}
b=\are_{T,S}(1-)=\tfrac{f}{g}(1-)
  =\tfrac{f'}{g'}(1-)
  =\tfrac{f''(1-)}{g''(1-)}
  =\tfrac{99\sqrt{15}-135\sqrt3}{40\pi}
\end{equation}
and
\begin{equation}\label{eq:lHosp.c}
\begin{split}
c&=\are_{T,S}'(1-)=\tfrac{f'g-fg'}{g^2}(1-)
  =\tfrac{f''g-fg''}{2gg'}(1-)
  =\tfrac{f'''g+f''g'-f'g''-fg'''}{2(g')^2+2gg''}(1-)
\\
& =\tfrac{f^{(4)}g+2f'''g'-2f'g'''-fg^{(4)}}{6g'g''+2gg'''}(1-)
  =\tfrac{f'''(1-)g''(1-)-f''(1-)g'''(1-)}{3g''(1-)^2}. 
\end{split}
\end{equation}
Then $f_9(0+)=g_9(0+)=0$ (and $f_{10}>0$, $g_{10}>0$, by Lemma~\ref{lem:q_TS}) imply that $f_9>0$, $g_9>0$ and $\rho_8=r_9\,\inc$ (by the special-case rules). 
Also, $f_8(0+)>0$, $g_8(1-)<0$ and $\tilde\rho_8(0+)<0$ imply $f_8>0$, $g_8<0$, and (by the refined general rules) $\rho_7=r_8\,\dec$ on $(0,1)$. 

Next, $g_7(0+)>0>g_7(1-)$ implies the existence of a single root $x_7$; $f_7(0+)>0$ shows that $f_7>0$. 
The refined general rules imply $\tilde\rho_7\,\inc$ on $(0,x_7)$ and $\dec$ on $(x_7,1)$. 
As $\tilde\rho_7(0+)>0$, we see $r_7\,\inc$ on $(0,x_7)$; 
further, $x_7<0.2$ (implied by $g_7(0.2)<0$) yields $\tilde\rho_7(x_7+)>\tilde\rho_7(0.2)>0>\tilde\rho_7(1-)$, so that $r_7\,\inc\,\dec$ on $(x_7,1)$. 
That is, $\rho_6=r_7\,\inc$ on both of $(0,x_7)$ and $(x_7,z_7)$, and $\rho_6=r_7\,\dec$ on $(z_7,1)$.

Next, $g_6(0+)>0>g_6(1-)$ implies the existence of $x_6$; $f_6(0+)>0$ implies $f_6>0$ on $(0,1)$. 
As $\tilde\rho_6(0+)>0$, the refined general rules imply $r_6\,\inc$ on $(0,x_7)$. 
Further, $g_6(0.5)>0>r_7'(0.5)$ implies $z_7<0.5<x_6$; 
as $f_6>0$, $f_6'>0$, $g_6>0$, and $g_6'<0$ on the interval $(x_7,x_6)$, we have $r_6'=\frac{f_6'g_6-f_6g_6'}{g_6^2}>0$ and hence $r_6\,\inc$ on $(x_7,x_6)$, so that $r_6\,\inc$ on $(0,x_6)$ (since $r_6$ is continuous at $x_7$). 
Also, $\tilde\rho_6\,\dec$ on $(x_6,1)$ is implied by the refined general rules; 
then $g_6(0.85)<0$ implies $x_6<0.85$, so that $\tilde\rho_6(x_6+)>\tilde\rho_6(0.85)>0>\tilde\rho_6(1-)$ shows that $r_6\,\inc\dec$ on $(x_6,1)$. 
That is, $\rho_5=r_6\,\inc$ on $(0,x_6)$ and $(x_6,z_6)$ and $\dec$ on $(z_6,1)$.

Next, $g_5(0+)>0$ and $g_5(1-)>0$, along with $g_5\,\inc\,\dec$, imply $g_5>0$ on $(0,1)$; 
also, $f_5(0+)<0<f_5(1-)$ implies $f_5$ has a single root $y_5$. 
The refined general rules imply $r_5\,\inc$ on $(0,x_6)$, as $\tilde\rho_5(0+)>0$; 
also, $f_5(0.5)>0$ implies $y_5<0.5<x_6$, so that $f_5>0$, $f_5'>0$, $g_5>0$ and $g_5'<0$ on $(x_6,1)$, and hence $r_5'=\frac{f_5'g_5-f_5g_5'}{g_5^2}>0$ on $(x_6,1)$. 
As $r_5$ is continuous at $x_6$, one has $\rho_4=r_5\,\inc$ on $(0,1)$.

Next, $-\infty=g_4(0+)<0<g_4(1-)$ shows $g_4$ has a single root $x_4$; 
$f_4(0+)=\infty>0>f_4(0.75)$ and $f_4(1-)>0$ shows $f_4$ has two roots $y_{4,1}$ and $y_{4,2}$. 
Also, $g_4(0.75)<0<g_4(0.8)$, $f_4(0.75)<0$, and $f_4(0.8)<0$ together imply $x_4\in(0.75,0.8)\subset(y_{4,1},y_{4,2})$, so that $r_4(x_4-)=\infty$ and $r_4(x_4+)=-\infty$. 
The general rules then imply $\rho_3=r_4\,\inc$ on each of $(0,x_4)$ and $(x_4,1)$.

Next, $g_3(0+)>0=g_3(1-)$ and $g_3\,\dec\,\inc$ shows $g_3$ has a single root $x_3$; 
$f_3(0+)>0=f_3(1-)$ and $f_3\,\inc\,\dec\,\inc$ shows $f_3$ has a single root $y_3$. 
Then $r_3\,\inc$ on $(x_4,1)$ by the special-case rules; 
$g_3(0.5)<0<f_3(0.5)$ yields $x_3<y_3$ (and hence $r_3(x_3-)=\infty$ and $r_3(x_3+)=-\infty$), so that the general rules imply $r_3\,\inc$ on both of $(0,x_3)$ and $(x_3,x_4)$. 
As $r_3$ is continuous at $x_4$, $\rho_2=r_3\,\inc$ on $(0,x_3)$ and $(x_3,1)$.

Next, $g_2(0+)<0=g_2(1-)$ and $f_2(0+)<0=f_2(1-)$ together yield the existence of roots $x_2$ and $y_2$, along with $r_2\,\inc$ on $(x_3,1)$ (via the special-case rules). 
Also, $g_2(0.1)>0>f_2(0.1)$ implies $x_2<y_2$ (and hence $r_2(x_2-)=\infty$ and $r_2(x_2+)=-\infty$), so that the general rules then imply $r_2\,\inc$ on $(0,x_2)$ and $(x_2,x_3)$. Further, $r_2$ is continuous at $x_3$ and hence $\rho_1=r_2\,\inc$ on $(0,x_2)$ and $(x_2,1)$.

Next, $g_1(0+)<0=g_1(1-)$ and $f_1(0+)<0=f_1(1-)$ show that $g_1<0$ and $f_1<0$ on $(0,1)$, and also $r_1\,\inc$ on $(x_2,1)$ by the special-case rules; 
$\tilde\rho_1(0+)>0$ implies via the refined general rules that $r_1\,\inc$ on $(0,x_2)$. 
The continuity of $r_1$ at $x_2$ then shows $\rho_0=r_1\,\inc$ on $(0,1)$.

Lastly, $f_0(1-)=g_0(1-)=0$ shows that $g_0>0$ and further, via the special-case rules, that $q_{T,S;1}=r_0\,\inc$ on $(0,1)$. 
Note $r_0(0+)=\frac{f_0(0+)}{g_0(0+)}$ and, by the l'Hospital rule for limits, $r_0(1-)=r_4(1-)$.
\end{proof}

\begin{proof}[Proof of Theorem~\ref{thm:quad}, (RS0)]
See Appendix~RS0 for more details of the following arguments. 
Set $a=0$ in the notation of Lemma~\ref{lem:q_RS}, so that, in accordance with \eqref{eq:b,c}, $b=\are_{R,S}(0)=\frac{\pi^2}9$ and $c=\are_{R,S}'(0)=0$. 
Then $f_4(0+)=g_4(0+)=0$, $f_5>0$, and $g_5>0$ (from Lemma~\ref{lem:q_RS}) together imply that $f_4>0$, $g_4>0$, and also $\rho_3=r_4\,\inc$ (via the special-case rules).

Next, $g_3\,\inc$ and $g_3(0+)<0<g_3(1-)$ implies the existence of the root $x_3$; that $f_3$ has a single root $y_3$ follows by $f_3\,\inc$ and $f_3(0+)<0<f_3(1-)$. 
From $x_3<y_3$ (implied by $g_3(0.64)>0>f_3(0.64)$) follows $r_3(x_3-)=\infty$ and $r_3(x_3+)=-\infty$; 
the general rules then imply $\rho_2=r_3\,\inc$ on both $(0,x_3)$ and $(x_3,1)$.

That $g_2$ has two distinct roots $x_{2,1}$ and $x_{2,2}$ follows from $g_2(0+)>0>g_2(0.5)$ and $g_2(1-)>0$ (along with $g_2\,\dec\,\inc$); 
similarly, $f_2$ has two roots $y_{2,1}$ and $y_{2,2}$, which follows from $f_2(0+)>0>f_2(0.5)$ and $f_2(1-)>0$. 
Then $g_2(0.33)<0<f_2(0.33)$ shows that $x_{2,1}<y_{2,1}$, and $g_2(0.86)>0>f_2(0.86)$ (together with $0>g_2(0.5)$ and  $0>f_2(0.5)$) show that $y_{2,1}<0.5<x_{2,2}<y_{2,2}$. 
The general rules then imply (since $r_2(x_{2,1}-)=r_2(x_{2,2}-)=\infty$ and $r_2(x_{2,1}+)=r_2(x_{2,2}+)=-\infty$) that $\rho_1=r_2\,\inc$ on the four intervals $(0,x_{2,1})$, $(x_{2,1},x_3)$, $(x_3,x_{2,2})$ and $(x_{2,2},1)$; 
the continuity of $r_2$ at $x_3$ implies $\rho_1=r_2\,\inc$ on $(x_{2,1},x_{2,2})$.

As $f_1(0+)=g_1(0+)=f_1(1-)=g_1(1-)=0$, one finds the existence of roots $x_1$ and $y_1$ (since $g_2$ and $f_2$ are both $+-+$), as well as $r_1\,\inc$ on $(0,x_{2,1})$ and $(x_{2,2},1)$ via the special-case rules. 
Further, $g_1(0.6)<0<f_1(0.6)$ shows $x_1<y_1$ (and hence $r_1(x_1-)=\infty$ and $r_1(x_1+)=-\infty$), so that the general rules imply $r_1\,\inc$ on $(x_{2,1},x_1)$ and $(x_1,x_{2,2})$. 
The continuity of $r_1$ at $x_{2,1}$ and $x_{2,2}$ then implies $\rho_0=r_1\,\inc$ on $(0,x_1)$ and $(x_1,1)$.

Lastly, $f_0(0+)=g_0(0+)=f_0(1-)=g_0(1-)=0$ and $g_0\,\inc\,\dec$ imply that $g_0>0$ and also (by the special-case rules) that $r_0\,\inc$ on both  $(0,x_1)$ and $(x_1,1)$. 
The continuity of $r_0$ at $x_1$ then implies $q_{R,S;0}=r_0\,\inc$ on $(0,1)$. 
The l'Hospital rule for limits implies $r_0(0+)=r_2(0+)$ and $r_0(1-)=r_2(1-)$.
\end{proof}

\begin{proof}[Proof of Theorem~\ref{thm:quad}, (RS1)]
See Appendix~RS1 for more details of the following arguments. 
Adopt the notation of Lemma~\ref{lem:q_RS}, with $a=1$, so that $f(1-)=g(1-)=f'(1-)=g'(1-)=0$ and repeated application of the l'Hospital rule for limits together yield (similar to \eqref{eq:lHosp.b} and \eqref{eq:lHosp.c})
\begin{equation*}
b=\are_{R,S}(1-)=\tfrac{f''(1-)}{g''(1-)}=\tfrac{3(11\sqrt5-15)}{20}
\quad\text{and}\quad
c=\are_{R,S}'(1-)=\tfrac{f'''(1-)g''(1-)-f''(1-)g'''(1-)}{3g''(1-)^2}.
\end{equation*}

From $g_4(0+)<0<g_4(1-)$ and $g_5>0$ follows the existence of $x_4$; 
similarly, $f_4(0+)<0<f_4(1-)$ and $f_5>0$ imply the existence of $y_4$. 
Then $g_4(0.8)>0>f_4(0.8)$ shows $x_4<0.8<y_4$, or hence $r_4(x_4-)=\infty$ and $r_4(x_4+)=-\infty$, and so the general rules imply $\rho_3=r_4$ $\inc$ on both $(0,x_4)$ and $(x_4,1)$.

Next, $g_3\,\dec\inc$ (as $g_4$ is $-+$) and $g_3(0+)=\infty>0=g_3(1-)$ yield the existence of $x_3$; 
that $f_3$ has a single root $y_3$ also follows by $f_3\,\dec\inc$ and $f_3(0+)=\infty>0=f_3(1-)$. 
The special-case rules imply $r_3\,\inc$ on $(x_4,1)$; also $x_3<y_3$ follows from $g_3(0.5)<0<f_3(0.5)$ (whence $r_3(x_3-)=\infty$ and $r_3(x_3+)=-\infty$), and so, the general rules imply $r_3\,\inc$ on both of $(0,x_3)$ and $(x_3,x_4)$. 
Also, $r_3$ is continuous at $x_4$ and hence $\rho_2=r_3\,\inc$ on $(0,x_3)$ and $(x_3,1)$.

As $g_2(0+)<0=g_2(1-)$ and $f_2(0+)<0=f_2(1-)$ (and $g_3$ and $f_3$ are both $+-$), there exist roots $x_2$ and $y_2$; the special-case rules imply $r_2\,\inc$ on $(x_3,1)$. 
Further, $g_2(0.1)>0>f_2(0.1)$ shows $x_2<0.1<y_2$ and hence $r_2(x_2-)=\infty$ and $r_2(x_2+)=-\infty$. 
The general rules then imply $r_2\,\inc$ on $(0,x_2)$ and $(x_2,x_3)$; 
the continuity of $r_2$ at $x_3$ then implies $\rho_1=r_2\,\inc$ on $(0,x_2)$ and $(x_2,1)$.

One finds that $g_1<0$ and $f_1<0$ on $(0,1)$, as $g_1(0+)<0=g_1(1-)$ (with $g_1\,\dec\inc$) and $f_1(0+)<0=f_1(1-)$ (with $f_1\,\dec\inc$), which further imply by the special-case rules that $r_1\,\inc$ on $(x_2,1)$. 
Also, $\tilde\rho_1(0+)>0$ implies via the refined general rules that $\tilde\rho_1>0$, or $r_1\,\inc$, on $(0,x_2)$; 
as $r_1$ is continuous on $(0,1)$, one sees $\rho_0=r_1\,\inc$ on $(0,1)$. 

Lastly, $f_0(1-)=g_0(1-)=0$ imply in the first place that $g_0>0$ (as $g_0\,\dec$), and in the second place that $q_{R,S;1}=r_0\,\inc$ on $(0,1)$ (via the special-case rules). 
The l'Hospital rule for limits implies $r_0(1-)=r_4(1-)$, and $g_0(0+)>0$ implies $r_0(0+)=\frac{f_0(0+)}{g_0(0+)}$.
\end{proof}

\begin{proof}[Proof of Corollary~\ref{cor:mon}]
As the $\are$'s are even functions here, one has $\are'(0)=0$, and hence $\are(x)=\are(0)+x^2q_0(x)$ for $x\in(0,1)$. 
Theorem~\ref{thm:quad} shows $q_0\,\inc$ and $q_0(0+)>0$, which imply $q_0>0$ on $(0,1)$; hence $\are\,\inc$ on $(0,1)$ as well. 
The values $\are(0+)$ and $\are(1-)$ are exactly those values of $b$ given at the beginning of the proof of each of the six parts of Theorem~\ref{thm:quad}.
\end{proof}

\begin{proof}[Proof of Corollary~\ref{cor:bounds}]
The result immediately follows from Theorem~\ref{thm:quad}:
\begin{align*}
&(x-a)^2q_a(0+)<\are(x)-\are(a)-\are'(a)(x-a)<(x-a)^2q_a(1-)
\\&\qquad\Rightarrow\quad
L_a(x)<\are(x)<U_a(x)
\end{align*}
for all $x\in(0,1)$ and $a\in\{0,1\}$. 
Replacing ``$x$'' with ``$-x$'' in the above inequality when $x\in(-1,0)$ and recalling the $\are$ is even yields the desired results.
\end{proof}

\begin{proof}[Proof of Corollary~\ref{cor:quartic}]
Noting that $\are_{R,S}=\are_{R,T}\cdot\are_{T,S}$, one has $L_{R,T}\cdot L_{T,S}<\are_{R,S}<U_{R,T}\cdot U_{T,S}$ on $(-1,1)\setminus\{0\}$. 
That $\tilde L_{R,S}>L_{R,S}$ and $\tilde U_{R,S}<U_{R,S}$ is easily verifed by noting $\tilde L_{R,S}-L_{R,S}$ and $\tilde U_{R,S}-U_{R,S}$ have no roots on $(-1,1)\setminus\{0\}$ and verifying their appropriate signs.
\end{proof}



\section*{Supplementary material}

# Monotonicity properties of the asymptotic relative efficiency between common correlation statistics in the bivariate normal model

## Appendix RT

**Raymond Molzon and Iosif Pinelis**

Michigan Technological University

## RTr: Proof of Lemma 3.2 (Reduction of $q_{R,T;a}$ to an algebraic form)

Assume by default all functions are defined on (0, 1):

```
$Assumptions = 0 < x < 1;
```

The asymptotic mean and variance of Pearson's $R$ are:

```
μ_R[x_] = x; (* (3.2) in paper *)
(σ²)_R[x_] = (1 - x²)²; (* (3.1) in paper *)
```

The asymptotic mean and variance of Kendall's $T$ are:

```
μ_T[x_] = 2/π ArcSin[x]; (* (3.6) in paper *)

(σ²)_T[x_] = 4/9 - 16/π² ArcSin[x/2]²; (* (3.7) in paper *)
```

Express $ARE_{R,T}(x)$ as the ratio $ARE_{R,T} = \frac{f}{g}$:

```
f[x_] = π² - 36 ArcSin[x/2]²;

g[x_] = 9 (1 - x²);

ARE[x_] = ((σ²)_T[x] μ_R'[x]²) / ((σ²)_R[x] μ_T'[x]²);

ARE[x] == f[x]/g[x] // Simplify

True
```

This confirms that $ARE_{R,T} = \frac{f}{g}$.

Begin the reduction phase, letting $b = ARE_{R,T}(a)$, $c = ARE'_{R,T}(a)$ and $q_{R,T;a}(x) = \frac{f_0(x)}{g_0(x)} = r_0(x)$, for arbitrary $a \in [0, 1]$.

Follow the scheme outlined in the paper (cf. (3.11) there), wherein $f_i = a_i f'_{i-1}$, $g_i = a_i g'_{i-1}$, $r_i = \frac{f_i}{g_i} = \frac{f'_{i-1}}{g'_{i-1}} = \rho_{i-1}$, and $a_i > 0$ on (0, 1).

At each step, note the continuity of $f_i$ and $g_i$ in $x \in (0, 1)$.

```
f0[x_] = f[x] - b g[x] - c (x - a) g[x] // Simplify

π² + 9 b (-1 + x²) - 9 c (a - x) (-1 + x²) - 36 ArcSin[x/2]²

g0[x_] = (x - a)² g[x] // Simplify

-9 (a - x)² (-1 + x²)
```

```
r0[x_] = f0[x]/g0[x];
```

Note that *b* and *c* are left as unspecified constants in the reduction phase; they will be given specific values (upon a choice of *a*) in the "working backwards" phase of sections RT0 and RT1.

---

```
a1[x_] = √(4 - x²) ;
```

This choice of $a_1$ yields a constant coefficient on $\sin^{-1}\left(\frac{x}{2}\right)$ in $f_1$:

```
f1[x_] = Map[Simplify, Collect[a1[x] f0'[x], {ArcSin[_], Sqrt[_]}], 2]
```

$$-9 \sqrt{4 - x^2} \left(c - 2 b x + 2 a c x - 3 c x^2\right) - 72 \operatorname{ArcSin}\left[\frac{x}{2}\right]$$

```
g1[x_] = a1[x] g0'[x] // Simplify
```

$$-18 \sqrt{4 - x^2} \left(a - x + a^2 x - 3 a x^2 + 2 x^3\right)$$

```
r1[x_] = f1[x]/g1[x];
```

---

```
a2[x_] = √(4 - x²)/(2 - x²);
```

This choice of $a_2$ yields a constant coefficient on *b* in $f_2$:

```
f2[x_] = Map[Simplify, Collect[a2[x] f1'[x], {b, c, a}], 2]
```

$$36 b + \frac{72}{-2 + x^2} + \frac{9 c \left(8 a - 25 x - 4 a x^2 + 9 x^3\right)}{-2 + x^2}$$

```
g2[x_] = Map[Simplify, Collect[a2[x] g1'[x], a], 2]
```

$$-36 a^2 + \frac{18 a x \left(-25 + 9 x^2\right)}{-2 + x^2} - \frac{36 \left(2 - 13 x^2 + 4 x^4\right)}{-2 + x^2}$$

```
r2[x_] = f2[x]/g2[x];
```

Note that $r_2$ is a rational function, and indeed one could begin using common analytical techniques to deduce the roots and monotonicity patterns of $f_2$, $g_2$, and $r_2$.

We continue with the reduction phase, so as to obtain a ratio independent of the value of *a* (and hence *b* and *c*); this is done primarily for the sake of consistency with the reduction phase in Appendices TS and RS (wherein more complicated functions are considered).

---

```
a3[x_] = (2 - x²)²/(50 - 29 x² + 9 x⁴);
```

This choice of $a_3$ yields a constant coefficient on *c* in $f_3$ and *a* in $g_3$:

```
f3[x_] = Map[Simplify, Collect[a3[x] f2'[x], c], 2]
```

$$9 c - \frac{144 x}{50 - 29 x^2 + 9 x^4}$$

```
g3[x_] = Map[Simplify, Collect[a3[x] g2'[x], a], 2]
```

$$18 a - \frac{288 x \left(6 - 4 x^2 + x^4\right)}{50 - 29 x^2 + 9 x^4}$$

```
r3[x_] = f3[x]/g3[x];
```

```
a4[x_] = (50 - 29 x^2 + 9 x^4)^2 / (2 - x^2);
```

```
f4[x_] = a4[x] f3'[x] // Simplify
```

$-144 \left(25 + 27 x^2\right)$

```
g4[x_] = a4[x] g3'[x] // Simplify
```

$864 \left(-50 + 46 x^2 - 11 x^4 + 3 x^6\right)$

```
r4[x_] = f4[x] / g4[x];
```

Now $f_4$, $g_4$ and $r_4$ are all algebraic functions, independent of the value of $a$.

```
Reduce[f4[x] < 0 && 0 < x < 1]
```

$0 < x < 1$

```
Reduce[g4[x] < 0 && 0 < x < 1]
```

$0 < x < 1$

```
Reduce[r4'[x] > 0 && 0 < x < 1]
```

$0 < x < 1$

Thus, $f_4 < 0$, $g_4 < 0$ and $r_4' > 0$ on $(0, 1)$, so that $r_4 = \rho_3 = \frac{f_3'}{g_3'} \nearrow$ on $(0, 1)$.  Note this is true for arbitrary $a \in [0, 1]$.  This proves Lemma 3.2 in the paper.

## RT0: Monotonicity of $q_{R,T;0}$ on $(0, 1)$

Below are numerical calculations supporting the arguments of the proof of statement (RT0) in Theorem 2.1.

```
{a, b, c} = {0, ARE[0], ARE'[0]}
```

$\left\{0, \dfrac{\pi^2}{9}, 0\right\}$

```
{g3[0], f3[0]} // Simplify
```

{0, 0}

```
Sign /@ {g2[0], g2[41 / 100], g2[1]}
```

{1, -1, -1}

```
Sign /@ {f2[0], f2[41 / 100], f2[1]}
```

{1, 1, -1}

```
Sign /@ {g1[0], g1[71 / 100], g1[1]}
```

{0, -1, -1}

```
Sign /@ {f1[0], f1[71 / 100], f1[1]}
```

{0, 1, -1}

```
{g0[0], g0[1]}
```

{0, 0}

```
{f0[0], f0[1]}
```

{0, 0}

```
{r2[0] // Simplify, r2[0] // N}
```

$\left\{\dfrac{1}{9}\left(-9 + \pi^2\right), 0.0966227\right\}$

```
{r1[1] // Simplify, r1[1] // N}
```

$\left\{\dfrac{1}{9}\left(2\sqrt{3} - \pi\right)\pi, 0.112577\right\}$

## RT1: Monotonicity of $q_{R,T;1}$ on $(0, 1)$

Below are numerical calculations supporting the arguments of the proof of statement (RT1) in Theorem 2.1.

**{f[1], g[1]}**

{0, 0}

$$\textbf{\{a, b, c\}} = \left\{1, \frac{\textbf{f'[1]}}{\textbf{g'[1]}}, \frac{\textbf{f''[1] g'[1] - f'[1] g''[1]}}{\textbf{2 g'[1]}^2}\right\} \textbf{// Simplify}$$

$$\left\{1, \frac{2\pi}{3\sqrt{3}}, -\frac{2}{27}\left(-9 + \sqrt{3}\ \pi\right)\right\}$$

---

**Sign /@ {g3[0], g3[3/5], g3[1]}**

{1, -1, -1}

**Sign /@ {f3[0], f3[3/5], f3[1]}**

{1, 1, -1}

---

**{g2[0], g2[1]}**

{0, 0}

**Sign /@ {f2[0], f2[1] // Simplify}**

{-1, 0}

**Sign[r3[0] g2[0] - f2[0]]  (* $\tilde{\rho}_2$(0+)>0 *)**

1

---

**{g1[1], f1[1]}**

{0, 0}

---

**{g0[1], f0[1]}**

{0, 0}

---

**{r0[0] // Simplify, r0[0] // N}**

$$\left\{\frac{1}{27}\left(18 - 8\sqrt{3}\ \pi + 3\pi^2\right), 0.151023\right\}$$

**{r3[1] // Simplify, r3[1] // N}**

$$\left\{\frac{1}{81}\left(-9 + 5\sqrt{3}\ \pi\right), 0.224778\right\}$$

# Monotonicity properties of the asymptotic relative efficiency between common correlation statistics in the bivariate normal model

## Appendix TS

**Raymond Molzon and Iosif Pinelis**

Michigan Technological University

## TSr: Proof of Lemma 3.3 (Reduction of $q_{T,S;a}(x)$ to an algebraic form)

By default, all functions shall be assumed to be on (0, 1):

```
$Assumptions = 0 < x < 1
```

0 < x < 1

A term that shall recur several times is $\sin^{-1}(1/4) + 2\sin^{-1}\left(\sqrt{3/8}\right)$. Letting $\alpha = \sin^{-1}(1/4)$ and $\beta = \sin^{-1}\left(\sqrt{3/8}\right)$, note that

$$\sin(\alpha + 2\beta) = \sin(\alpha)\cos(2\beta) + \cos(\alpha)\sin(2\beta)$$
$$= \sin(\alpha)\left(1 - 2\sin^2(\beta)\right) + 2\cos(\alpha)\cos(\beta)\sin(\beta)$$
$$= \frac{1}{16} + \frac{15}{16} = 1.$$

Thus, $\alpha + 2\beta = \frac{\pi}{2} + 2\pi k$ for some natural number $k$; that $k = 0$ is easily numerically verified :

```
ArcSin[1/4] + 2 ArcSin[√(3/8)] // N
```

1.5708

```
π/2 // N
```

1.5708

```
sub = {ArcSin[1/4] → π/2 - 2 ArcSin[√(3/8)]};
```

The asymptotic mean and variance of Kendall's $T$ are:

```
μ_T[x_] = 2/π ArcSin[x];  (* (3.6) in the paper *)
```

```
(σ²)_T[x_] = 4/9 - 16/π² ArcSin[x/2]²;  (* (3.7) in the paper *)
```

The asymptotic mean and variance of Spearman's $S$ are:

```
μs[x_] = 6/π ArcSin[x/2]; (* (3.8) in the paper *)

(σ²)s[x_] := 1 - 324/π² ArcSin[x/2]² + 72/π² (I1[x] + 2 I2[x] + 2 I3[x] + 4 I4[x]); (* (3.9) in the paper *)

I1'[x_] = ArcSin[x³/(8-4 x²)]/√(4-x²); I2'[x_] = ArcSin[x/(6-2 x²)]/√(4-x²);

I3'[x_] = ArcSin[(x (4-x²))/(2√2 √(8-6 x²+x⁴))]/√(4-x²); I4'[x_] = ArcSin[(x (4-x²))/(2√(12-7 x²+x⁴))]/√(4-x²);

I1[x_?NumberQ] := NIntegrate[I1'[u], {u, 0, x}]
I2[x_?NumberQ] := NIntegrate[I2'[u], {u, 0, x}]
I3[x_?NumberQ] := NIntegrate[I3'[u], {u, 0, x}]
I4[x_?NumberQ] := NIntegrate[I4'[u], {u, 0, x}]
I1[0] = I2[0] = I3[0] = I4[0] = 0;

(σ²)s[1] = 0; (* This follows from Lemma 3.1 in the paper. *)
```

Express $\mathrm{ARE}_{T,S}$ as a ratio $\mathrm{ARE}_{T,S} = \frac{f}{g}$:

```
f[x_] := (σ²)s[x]

g[x_] = (4 (1 - x²) (π² - 36 ArcSin[x/2]²))/(π² (4 - x²));

ARE[x_] = ((σ²)s[x] μT'[x]²)/((σ²)T[x] μS'[x]²);

ARE[x] == f[x]/g[x] // Simplify

True
```

This confirms that $\mathrm{ARE}_{T,S} = \frac{f}{g}$.

---

Begin the reduction phase (cf. Appendix RTr and (3.11) in the paper), letting $b = \mathrm{ARE}_{T,S}(a)$, $c = \mathrm{ARE}'_{T,S}(a)$ and $q_{T,S}(x; a) = \frac{f_0(x)}{g_0(x)} = r_0(x)$, for arbitrary $a \in [0, 1]$.

```
f0[x_] := f[x] - b g[x] - c (x - a) g[x] // Simplify

g0[x_] = (x - a)² g[x] // Simplify;

r0[x_] = f0[x]/g0[x];
```

---

```
a1[x_] = √(4 - x²) ;
```

This choice of $a_1$ ensures all of the $\sin^{-1}(\cdot)$ terms associated with the integrals $I_1$, $I_2$, $I_3$, and $I_4$ (found in $f_0$) have constant coefficients.

```
f1[x_] = Map[Simplify, Collect[a1[x] f0'[x] // Together, ArcSin[_]], 2]

g1[x_] = Map[Simplify, Collect[a1[x] g0'[x] // Together, ArcSin[_]], 2]

r1[x_] = f1[x]/g1[x];
```

Some terms that will be convenient to collect in the functions defined from hereon:

```
s1 = √(16 - 12 x² + x⁴) √(12 - 8 x² + x⁴) √(9 - 4 x²) ;
s2 = √(16 - 12 x² + x⁴) √(12 - 8 x² + x⁴) ;
s3 = √(16 - 12 x² + x⁴) √(9 - 4 x²) ;
s4 = √(9 - 4 x²) √(12 - 8 x² + x⁴) ;
```

Below are some terms that will arise in the computation of $f_2$; they will be convenient to rewrite in terms of $s_1$, $s_2$, $s_3$, and $s_4$:

```
(t1 = √((-4 + x²)³ (-1728 + 3216 x² - 2204 x⁴ + 676 x⁶ - 89 x⁸ + 4 x¹⁰))) == (4 - x²)^(3/2) s1 // Simplify
```
True

```
(t2 = √(-(-4 + x²)³ (192 - 272 x² + 124 x⁴ - 20 x⁶ + x⁸))) == (4 - x²)^(3/2) s2 // Simplify
```
True

```
(t3 = √((-4 + x²)³ (-144 + 172 x² - 57 x⁴ + 4 x⁶))) == (4 - x²)^(3/2) s3 // Simplify
```
True

```
(t4 = √((-4 + x²)³ (-108 + 120 x² - 41 x⁴ + 4 x⁶))) == (4 - x²)^(3/2) s4 // Simplify
```
True

Then define the appropriate substitutions:

```
sub2 = {t1 → (4 - x²)^(3/2) s1, t2 → (4 - x²)^(3/2) s2, t3 → (4 - x²)^(3/2) s3, t4 → (4 - x²)^(3/2) s4};
```

```
a2[x_] = (4 - x²)^(5/2) / (2 + x²);
```

This choice of $a_2$ yields a constant coefficient on the term $b \left(\sin^{-1} \frac{x}{2}\right)^2$ in the function $f_2$.

```
f2[x_] = Map[Simplify,
   Collect[Map[Simplify, Collect[a2[x] f1'[x] // Simplify // Together, {ArcSin[_], s1, s2, s3, s4}], 2] /. sub2,
   {ArcSin[_], b, c, s1, s2, s3, s4}], 2]

g2[x_] = Map[Simplify, Collect[a2[x] g1'[x], ArcSin[_]], 2]

r2[x_] = f2[x] / g2[x];
```

```
h3[x_] = 38 - 17 x² - 3 x⁴;
a3[x_] = (2 + x²)² / ((4 - x²) h3[x]);
```

It will start becoming less obvious that the $a_i$'s are positive on $(0, 1)$, and so this will be routinely checked from hereon:

```
Reduce[a3[x] > 0 && 0 < x < 1]
```
0 < x < 1

This choice of $a_3$ yields a constant coefficient on $c \left(\sin^{-1} \frac{x}{2}\right)^2$ in $f_3$.

```
f3[x_] = Map[Simplify, Collect[a3[x] f2'[x] // Together, {ArcSin[_], b, c, s1, s2, s3, s4}], 2]
```

```
g3[x_] = Map[Simplify , Collect[a3[x] g2'[x], ArcSin[_]], 2]
```

$$r3[x\_] = \frac{f3[x]}{g3[x]};$$

Some further simplifying substitutions for $f_4$:

$$sub2 = Union\left[sub2, \left\{\left(\frac{-4 + x^2}{-9 + 4 x^2}\right)^{5/2} \rightarrow \frac{\left(4 - x^2\right)^{5/2}}{\left(9 - 4 x^2\right)^{5/2}}, \left(\frac{4 - x^2}{16 - 12 x^2 + x^4}\right)^{5/2} \rightarrow \frac{\left(4 - x^2\right)^{5/2}}{\left(16 - 12 x^2 + x^4\right)^{5/2}}\right\}\right];$$

```
h4[x_] = 892 - 440 x^2 + 61 x^4 - 9 x^6;
```

$$a4[x\_] = \frac{\left(4 - x^2\right)^{5/2} h3[x]^2}{x \left(2 + x^2\right) h4[x]};$$

```
Reduce[a4[x] > 0 && 0 < x < 1]
```

0 < x < 1

This choice of $a_4$ yields a constant coefficient on $b \sin^{-1}\left(\frac{x}{2}\right)$ in $f_4$.

```
f4[x_] = Map[Simplify, Collect[a4[x] f3'[x] // Together, {ArcSin[_], b, c, s1, s2, s3, s4}], 2] /. sub2
```

```
g4[x_] = Map[Factor, Map[Simplify, Collect[a4[x] g3'[x], ArcSin[_]], 2], 2]
```

$$r4[x\_] = \frac{f4[x]}{g4[x]};$$

```
h5[x_] = 328 256 - 60 276 x^2 - 28 380 x^4 + 12 853 x^6 - 678 x^8 + 81 x^10;
```

$$a5[x\_] = \frac{x^2 h4[x]^2}{\left(4 - x^2\right) h3[x] h5[x]};$$

```
Reduce[a5[x] > 0 && 0 < x < 1]
```

0 < x < 1

This choice of $a_5$ yields a constant coefficient on the term $c \sin^{-1}\left(\frac{x}{2}\right)$ in $f_5$.

```
f5[x_] = Map[Simplify, Collect[a5[x] f4'[x] // Together, {ArcSin[_], b, c, s1, s2, s3, s4}], 2]
```

```
g5[x_] = Map[Factor, Map[Simplify, Collect[a5[x] g4'[x], ArcSin[_]], 2]]
```

$$r5[x\_] = \frac{f5[x]}{g5[x]};$$

```
h6[x_] = 17 418 976 - 12 356 932 x^2 + 3 290 736 x^4 - 575 137 x^6 + 35 011 x^8 - 447 x^10 + 81 x^12;
```

$$a6[x\_] = \frac{\sqrt{4 - x^2} h5[x]^2}{x h4[x] h6[x]};$$

```
Reduce[a6[x] > 0 && 0 < x < 1]
```

0 < x < 1

This choice of $a_6$ yields a constant coefficient on $b$ in $f_6$; at this point, $f_6$ is an algebraic function.

```
f6[x_] = Map[Simplify,
   Collect[Map[Simplify, Collect[a6[x] f5'[x] // Together /. sub2, {b, c, Sqrt[_]}], 2] /. sub2 // Together,
    {b, c, s1, s2, s3, s4}], 2]
```

```
g6[x_] = Map[Factor, Map[Simplify, Collect[a6[x] g5'[x], ArcSin[_]], 2], 2]
```

$$r6[x\_] = \frac{f6[x]}{g6[x]};$$

```
h7[x_] =
    18 745 083 424 - 14 666 397 812 x^2 + 4 272 900 412 x^4 - 473 552 785 x^6 + 47 852 540 x^8 - 89 482 x^10 + 1296 x^12 - 729 x^14;
a7[x_] = h6[x]^2 / (h5[x] h7[x]);
Reduce[a7[x] > 0 && 0 < x < 1]

0 < x < 1
```

This choice of $a_7$ yields a constant coefficient on $c$ in the function $f_7$.

```
f7[x_] = Map[Simplify, Collect[a7[x] f6'[x] // Together, {c, s1, s2, s3, s4}], 2]

g7[x_] = Map[Simplify, Collect[a7[x] g6'[x] // Together, {ArcSin[_], Sqrt[_]}], 2]

r7[x_] = f7[x] / g7[x];
```

```
h8[x_] = 67 393 220 864 - 66 665 518 536 x^2 + 25 281 966 744 x^4 -
    4 783 210 446 x^6 + 320 370 996 x^8 - 26 281 941 x^10 - 170 777 x^12 + 231 x^14 + 81 x^16;
a8[x_] = h7[x]^2 / (h6[x] h8[x]);
Reduce[a8[x] > 0 && 0 < x < 1]

0 < x < 1
```

This choice of $a_8$ yields a constant coefficient on $\sin^{-1}\left(\frac{x}{2}\right)^2$ in the function $g_8$; note also that all dependence on the value of $a$ has vanished in $f_8$.

```
f8[x_] = Map[Simplify, Collect[a8[x] f7'[x] // Together, {s1, s2, s3, s4}], 2]

g8[x_] = Map[Simplify, Collect[a8[x] g7'[x], ArcSin[_]], 2]

r8[x_] = f8[x] / g8[x];
```

```
h9[x_] = 32 482 389 470 208 - 34 864 017 237 408 x^2 + 16 286 313 144 464 x^4 - 4 430 399 397 672 x^6 + 832 485 830 428 x^8 -
    100 457 826 796 x^10 + 7 855 470 828 x^12 - 362 114 966 x^14 + 14 054 393 x^16 + 127 203 x^18 + 31 x^20 - 9 x^22;
a9[x_] = ((4 - x^2)^(5/2) h8[x]^2) / (h7[x] h9[x]);
Reduce[a9[x] > 0 && 0 < x < 1]

0 < x < 1
```

The coefficient of $\sin^{-1}\left(\frac{x}{2}\right)$ in $g_9$ is now constant by this choice of $a_9$.

```
f9[x_] = Map[Simplify, Collect[a9[x] f8'[x] // Together, {s1, s2, s3, s4}], 2]

g9[x_] = Map[Simplify, Collect[a9[x] g8'[x], ArcSin[_]], 2]

r9[x_] = f9[x] / g9[x];
```

```
a10[x_] = (π^2 h9[x]^2) / (27 648 (4 - x^2)^(3/2) h8[x]);
Reduce[a10[x] > 0 && 0 < x < 1]

0 < x < 1
```

```
f10[x_] = Map[Factor, Collect[a10[x] f9'[x] // Together, {s1, s2, s3, s4}], 2]
```

```
g10[x_] = a10[x] g9'[x] // Together // Factor
```

$$r10[x_{\_}] = \frac{f10[x]}{g10[x]};$$

We now have $f_{10}$, $g_{10}$, and $r_{10}$ all algebraic functions independent of the value of $a$.

```
Reduce[g10[x] > 0 && 0 < x < 1]
```

```
0 < x < 1
```

```
d[x_] = f10'[x] g10[x] - f10[x] g10'[x] // Together;
```

```
d1[x_] = Map[Factor, Collect[Numerator[d[x]], {s1, s2, s3, s4}], 2]
```

```
d2[x_] = Denominator[d[x]]
```

```
Reduce[d2[x] > 0 && 0 < x < 1]
```

```
0 < x < 1
```

$$r10'[x] == \frac{d1[x]}{g10[x]^2 \, d2[x]} \; // \; Simplify$$

```
True
```

Then $r'_{10} = \frac{f'_{10} \, g_{10} - f_{10} \, g'_{10}}{g^2_{10}} = \frac{d_1}{d_2 \, g^2_{10}}$, where $d_2 > 0$ and $g_{10} > 0$; that is, $r'_{10} > 0$ if $d_1 > 0$.

```
p1 = Coefficient[d1[x], s1] // Factor
```

```
p2 = Coefficient[d1[x] - s1 p1, s2] // Factor
```

```
p3 = Coefficient[d1[x] - s1 p1, s3] // Factor
```

```
p4 = Coefficient[d1[x] - s1 p1, s4] // Factor
```

```
d1[x] == s1 p1 + s2 p2 + s3 p3 + s4 p4 // Simplify
```

```
True
```

Further, $d_1 = \sum_{i=1}^{4} s_i \, p_i$ is a function rational in the $s_i$'s.

```
Reduce[s1 p1 + s4 p4 > 0 && 0 < x < 1]
```

```
0 < x < 1
```

```
Reduce[s2 p2 > 0 && 0 < x < 1]
```

```
0 < x < 1
```

```
Reduce[s3 p3 > 0 && 0 < x < 1]
```

```
0 < x < 1
```

Then $d_1 > 0$, whence $r'_{10} > 0$ and $r_{10} = \rho_9 \nearrow$ on $(0, 1)$.

```
Sign[r10[0]]
```

```
1
```

So, $r_{10}(0 +) > 0$ and $r_{10} \nearrow$ imply $r_{10} > 0$ and hence $f_{10} > 0$ on $(0, 1)$. Note that this is true for arbitrary $a \in [0, 1]$. This proves Lemma 3.3.

## TS0: Monotonocity of $q_{T,S;0}$ on (0, 1)

Following are numerical calculations supporting the arguments of the proof of statement (TS0) in Theorem 2.1.

**{a, b, c} = {0, ARE[0], ARE'[0]}**

{0, 1, 0}

---

**{g9[0], f9[0]}**

{0, 0}

---

**Sign /@ {g8[1], f8[0]}**

{-1, 1}

**Sign[r10[0] g8[0] - f8[0]] (\* $r_9$(0+)=$r_{10}$(0+) by l'Hospital, so that $\tilde{\rho}_8$(0+)<0 \*)**

-1

---

**{g7[0], f7[0]}**

{0, 0}

---

**Sign /@ {g6[0], g6[1/2], g6[3/4], g6[1]}**

{1, 1, -1, -1}

**Sign[f6[0]]**

1

**Sign[-(r8[0] g6[0] - f6[0])] (\* $r_7$(0+)=$r_8$(0+) by l'Hospital, so that $\tilde{\rho}_6$(0+)>0 \*)**

1

**$\tilde{\rho}6$[x\_] = -(r7[x] g6[x] - f6[x]);**
**Sign /@ $\left\{\tilde{\rho}6[3/4], \tilde{\rho}6[1]\right\}$**

{1, -1}

---

**Sign /@ {g5[0], g5[1]}**

{1, 1}

**Sign[f5[0]]**

1

**$\tilde{\rho}5$[x\_] = r6[x] g5[x] - f5[x];**
**Sign /@ $\left\{\tilde{\rho}5[0], \tilde{\rho}5[1/2]\right\}$**

{-1, 1}

---

**Sign /@ {Limit[g4[x], x→0], g4[3/10], g4[2/5], g4[1]}**

{-1, -1, 1, 1}

**Sign /@ {Limit[f4[x], x→0], f4[2/5], f4[1]}**

{-1, -1, 1}

```
Sign[r5'[3/10]]
```

```
1
```

```
ρ̃4[x_] = r5[x] g4[x] - f4[x] // Together;
q1[x_] = Collect[Numerator[ρ̃4[x]], {ArcSin[_], Sqrt[_]}];
q2[x_] = Denominator[ρ̃4[x]];
{q1[0], q2[0]} // Simplify
```

```
{0, 0}
```

```
Sign /@ {q1'[0], q2'[0]}
```

```
{0, 1}
```

The l'Hospital rule shows that $\tilde{\rho}_4(0+) = \frac{q_1}{q_2}(0+) = \frac{q_1'(0+)}{q_2'(0+)} = 0$.

```
Sign /@ {g3[0], g3[16/25], g3[1]}
```

```
{0, 1, 1}
```

```
Sign /@ {f3[0], f3[16/25], f3[1]}
```

```
{0, -1, 1}
```

```
Sign /@ {g2[0], g2[7/20], g2[1/2], g2[43/50], g2[1]}
```

```
{1, -1, -1, 1, 1}
```

```
Sign /@ {f2[0], f2[7/20], f2[1/2], f2[43/50], f2[1]}
```

```
{1, 1, -1, -1, 1}
```

```
Sign /@ {g1[0], g1[31/50], g1[1]}
```

```
{0, -1, 0}
```

```
Sign /@ {f1[0], f1[31/50], f1[1] /. sub // Simplify}
```

```
{0, 1, 0}
```

```
{g0[0], g0[1]}
```

```
{0, 0}
```

```
{f0[0], f0[1]}
```

```
{0, 0}
```

```
{r2[0] // Simplify, r2[0] // N}
```

$$\left\{ \frac{3\left(32\left(-2+\sqrt{3}\right)+\pi^2\right)}{4\pi^2}, \ 0.0984257 \right\}$$

```
{r2[1] // Simplify, r2[1] // N}
```

$$\left\{ -\frac{135\sqrt{3} - 99\sqrt{15} + 40\pi}{40\pi}, \ 0.190467 \right\}$$

## TS1: Monotonocity of $q_{T,S;1}$ on $(0, 1)$

Following are numerical calculations supporting the arguments of the proof of statement (TS1) in Theorem 2.1.

```
{f[1], g[1], f'[1], g'[1]} /. sub // Simplify
```

{0, 0, 0, 0}

```
{a, b, c} = {1, f''[1]/g''[1], (f'''[1] g''[1] - f''[1] g'''[1])/(3 g''[1]^2)} /. sub // Simplify
```

$\left\{1, \dfrac{9\sqrt{3}\left(-15 + 11\sqrt{5}\right)}{40\pi}, -\dfrac{3\left(-135 + 99\sqrt{5} + 5\sqrt{3}\left(-9 + \sqrt{5}\right)\pi\right)}{40\pi^2}\right\}$

```
{g9[0], f9[0]}
```

{0, 0}

```
Sign /@ {g8[1], f8[0]}
```

{-1, 1}

```
Sign[r10[0] g8[0] - f8[0]](* r_9(0+)=r_10(0+) by l'Hospital, so that \tilde{\rho}_8(0+)<0 *)
```

-1

```
Sign /@ {g7[0], g7[1/5], g7[1]}
```

{1, -1, -1}

```
Sign[f7[0]]
```

1

```
\tilde{\rho}7[x_] = -(r8[x] g7[x] - f7[x]);
Sign /@ {\tilde{\rho}7[0], \tilde{\rho}7[1/5], \tilde{\rho}7[1]}
```

{1, 1, -1}

```
Sign /@ {g6[0], g6[1/2], g6[17/20], g6[1]}
```

{1, 1, -1, -1}

```
Sign[f6[0]]
```

1

```
Sign[r7[0] g6[0] - f6[0]] (* \tilde{\rho}_6(0+)>0 *)
```

1

```
Sign[r7'[1/2]]
```

-1

```
\tilde{\rho}6[x_] = -(r7[x] g6[x] - f6[x]);
Sign /@ {\tilde{\rho}6[17/20], \tilde{\rho}6[1]}
```

{1, -1}

```
Sign /@ {g5[0], g5[1]}
```

{1, 1}

```
Sign /@ {f5[0], f5[1/2], f5[1]}
```

{-1, 1, 1}

```
Sign[r6[0] g5[0] - f5[0]]  (* ρ̃₅(0+)>0 *)
```

1

```
Sign /@ {Limit[g4[x], x → 0], g4[3/4], g4[4/5], g4[1]}
```

{-1, -1, 1, 1}

```
Sign /@ {Limit[f4[x], x → 0], f4[3/4], f4[4/5], f4[1]}
```

{1, -1, -1, 1}

```
Sign /@ {g3[0], g3[1/2], g3[1]} // Simplify
```

{1, -1, 0}

```
Sign /@ {f3[0], f3[1/2], f3[1]} // Simplify
```

{1, 1, 0}

```
Sign /@ {g2[0], g2[1/10], g2[1]}
```

{-1, 1, 0}

```
Sign /@ {f2[0], f2[1/10], f2[1]} // Simplify
```

{-1, -1, 0}

```
Sign /@ {g1[0], g1[1]}
```

{-1, 0}

```
Sign /@ {f1[0], f1[1] /. sub // Simplify}
```

{-1, 0}

```
Sign[-(r2[0] g1[0] - f1[0])]  (* ρ̃₁(0+)>0 *)
```

1

```
{g0[1], f0[1]}
```

{0, 0}

```
{r0[0] // Simplify, r0[0] // N}
```

$$\left\{ \frac{405 - 297\sqrt{5} - 6\sqrt{3}\left(-45 + 19\sqrt{5}\right)\pi + 40\pi^2}{40\pi^2},\ 0.551626 \right\}$$

```
{r4[1] // Simplify, r4[1] // N}
```

$$\left\{ \frac{1350\sqrt{3}\left(-15 + 11\sqrt{5}\right) - 225\left(75 + \sqrt{5}\right)\pi + \sqrt{3}\left(-1125 + 4297\sqrt{5}\right)\pi^2}{2000\pi^3},\ 1.82004 \right\}$$

# Monotonicity properties of the asymptotic relative efficiency between common correlation statistics in the bivariate normal model

## Appendix RS

**Raymond Molzon and Iosif Pinelis**

Michigan Technological University

## RSr: Proof of Lemma 3.4 (Reduction of $q_{R,S;a}$ to an algebraic form)

By default, all functions shall be assumed to be on $(0, 1)$:

```
$Assumptions = 0 < x < 1
```

```
0 < x < 1
```

Use the same substitution as given in Appendix TSr:

```
sub = {ArcSin[1/4] → π/2 - 2 ArcSin[√(3/8)]};
```

The asymptotic mean and variance of Pearson's $R$:

```
μR[x_] = x; (* (3.2) in paper *)
```

```
(σ²)R[x_] = (1 - x²)²; (* (3.1) in paper *)
```

The asymptotic mean and variance of Spearman's $S$:

```
μS[x_] = 6/π ArcSin[x/2]; (* (3.8) in paper *)
```

```
(σ²)S[x_] := 1 - 324/π² ArcSin[x/2]² + 72/π² (I1[x] + 2 I2[x] + 2 I3[x] + 4 I4[x]); (* (3.9) in paper *)
```

```
I1'[x_] = ArcSin[x³/(8-4x²)] / √(4-x²);   I2'[x_] = ArcSin[x/(6-2x²)] / √(4-x²);
```

```
I3'[x_] = ArcSin[x(4-x²)/(2√2 √(8-6x²+x⁴))] / √(4-x²);   I4'[x_] = ArcSin[x(4-x²)/(2√(12-7x²+x⁴))] / √(4-x²);
```

```
I1[x_?NumberQ] := NIntegrate[I1'[u], {u, 0, x}]
I2[x_?NumberQ] := NIntegrate[I2'[u], {u, 0, x}]
I3[x_?NumberQ] := NIntegrate[I3'[u], {u, 0, x}]
I4[x_?NumberQ] := NIntegrate[I4'[u], {u, 0, x}]
I1[0] = I2[0] = I3[0] = I4[0] = 0;
```

```
(σ²)S[1] = 0; (* This follows from Lemma 3.1 in the paper. *)
```

Express $ARE_{R,S}$ as a ratio $ARE_{R,S} = \frac{f}{g}$:

```
f[x_] := (σ²)_S [x]
```

```
g[x_] = 36 (1 - x²)² / (π² (4 - x²)) ;
```

```
ARE[x_] = (σ²)_S [x] μ_R '[x]² / ((σ²)_R [x] μ_S '[x]²) ;
```

```
ARE[x] == f[x] / g[x] // Simplify
```

```
True
```

Begin the reduction phase (cf. (3.11) in the paper):

```
f0[x_] := f[x] - b g[x] - c (x - a) g[x] // Simplify
```

```
g0[x_] = (x - a)² g[x] // Simplify
```

```
r0[x_] = f0[x] / g0[x] ;
```

```
a1[x_] = √(4 - x²) ;
```

```
f1[x_] = Map[Simplify, Collect[a1[x] f0'[x] // Together, {ArcSin[_]}], 2]
```

```
g1[x_] = a1[x] g0'[x] // Simplify
```

```
r1[x_] = f1[x] / g1[x] ;
```

The $s_i$ below are terms that will be convenient to collect in the functions defined hereafter:

```
s1 = √(9 - 4 x²) √(16 - 12 x² + x⁴) √(12 - 8 x² + x⁴) ;
s2 = √(9 - 4 x²) √(12 - 8 x² + x⁴) ;
s3 = √(9 - 4 x²) √(16 - 12 x² + x⁴) ;
s4 = √(12 - 8 x² + x⁴) √(16 - 12 x² + x⁴) ;
```

```
(t1 = √(1728 - 3216 x² + 2204 x⁴ - 676 x⁶ + 89 x⁸ - 4 x¹⁰)) == s1 // Simplify
```

```
True
```

```
(t2 = √(108 - 120 x² + 41 x⁴ - 4 x⁶)) == s2 // Simplify
```

```
True
```

```
(t3 = √(144 - 172 x² + 57 x⁴ - 4 x⁶)) == s3 // Simplify
```

```
True
```

```
(t4 = √(192 - 272 x² + 124 x⁴ - 20 x⁶ + x⁸)) == s4 // Simplify
```

```
True
```

```
sub2 = {t1 → s1, t2 → s2, t3 → s3, t4 → s4};
```

```
a2[x_] = (4 - x^2)^(5/2);

f2[x_] = Map[Simplify,
    Collect[Map[Simplify, Collect[a2[x] f1'[x] // Simplify // Together /. sub2, {b, c, Sqrt[_]}], 2] /. sub2,
      {b, c, s1, s2, s3, s4}], 2]

g2[x_] = Map[Simplify, Collect[a2[x] g1'[x], a], 2]

r2[x_] = f2[x]/g2[x];
```

At this point, $f_2$, $g_2$, and hence $r_2$ are all algebraic functions.

```
h3[x_] = x (41 - 20 x^2 + 3 x^4);

a3[x_] = 1/h3[x];

Reduce[a3[x] > 0 && 0 < x < 1]

0 < x < 1
```

With this choice of $a_3$, the coefficient of $b$ in $f_3$ is constant.

```
f3[x_] = Map[Simplify, Collect[a3[x] f2'[x] // Together, {b, c, s1, s2, s3, s4}], 2]

g3[x_] = Map[Simplify, Collect[a3[x] g2'[x], a], 2]

r3[x_] = f3[x]/g3[x] /. sub;
```

```
h4[x_] = 7052 + 30147 x^2 - 35490 x^4 + 13432 x^6 - 2370 x^8 + 189 x^10;

a4[x_] = h3[x]^2/h4[x];

Reduce[a4[x] > 0 && 0 < x < 1]

0 < x < 1
```

The coefficient of $c$ in $f_4$ is constant with this choice of $a_4$.

```
f4[x_] = Map[Simplify, Collect[a4[x] f3'[x] // Together, {c, s1, s2, s3, s4}], 2]

g4[x_] = Map[Simplify, Collect[a4[x] g3'[x], a], 2]

r4[x_] = f4[x]/g4[x];
```

```
a5[x_] = π^2 h4[x]^2/(1728 x h3[x]);

Reduce[a5[x] > 0 && 0 < x < 1]

0 < x < 1

f5[x_] = Map[Simplify, Collect[a5[x] f4'[x] // Together, {s1, s2, s3, s4}], 2]

g5[x_] = a5[x] g4'[x] // Factor

r5[x_] = f5[x]/g5[x];
```

Now $f_5$, $g_5$, and $r_5$ are all independent of any choice of $a$.

```
Reduce[g5[x] > 0 && 0 < x < 1]
```

```
0 < x < 1
```

```
d[x_] = f5'[x] g5[x] - f5[x] g5'[x] // Together;
```

```
d1[x_] = Map[Factor, Collect[Numerator[d[x]], {s1, s2, s3, s4}], 2]
```

```
d2[x_] = Denominator[d[x]]
```

```
Reduce[d2[x] > 0 && 0 < x < 1]
```

```
0 < x < 1
```

$$r5'[x] == \frac{d1[x]}{d2[x]\ g5[x]^2}\ //\ \texttt{Simplify}$$

```
True
```

Hence $r_5' = \frac{f_5'\,g_5 - f_5\,g_5'}{g_5^2} = \frac{d_1}{d_2\,g_5^2}$, where $d_2 > 0$ and $g_5 > 0$ on $(0, 1)$.

```
p1 = Coefficient[d1[x], s1] // Factor
```

```
p2 = Coefficient[d1[x] - s1 p1, s2] // Factor
```

```
p3 = Coefficient[d1[x] - s1 p1, s3] // Factor
```

```
p4 = Coefficient[d1[x] - s1 p1, s4] // Factor
```

```
d1[x] == p1 s1 + p2 s2 + p3 s3 + p4 s4 // Simplify
```

```
True
```

Further, $d_1 = \sum_{i=1}^{4} s_i\,p_i$ is a rational function in the $s_i$'s.

```
Reduce[p1 s1 > 0 && 0 < x < 1]
```

```
0 < x < 1
```

```
Reduce[p2 s2 > 0 && 0 < x < 1]
```

```
0 < x < 1
```

```
Reduce[p3 s3 > 0 && 0 < x < 1]
```

```
0 < x < 1
```

```
Reduce[p4 s4 > 0 && 0 < x < 1]
```

```
0 < x < 1
```

Then $d_1 = \sum_{i=1}^{4} s_i\,p_i > 0$ on $(0, 1)$, whence $r_5 = \rho_4 \nearrow$ on $(0, 1)$.

```
Sign[r5[0]]
```

```
1
```

As $r_5(0+) > 0$, one has $r_5 > 0$ and hence $f_5 > 0$. Note that this fact is independent of the choice of $a \in [0, 1]$. Thus, Lemma 3.4 is proved.

## RS0: Monotonicity of $q_{R,S;0}$ on (0, 1)

Following are numerical calculations supporting the arguments of the proof of statement (RS0) in Theorem 2.1.

```
{a, b, c} = {0, ARE[0], ARE'[0]}
```

$\left\{0, \dfrac{\pi^2}{9}, 0\right\}$

---

```
{g4[0], f4[0]}
```

{0, 0}

---

```
Sign /@ {g3[x] /. x → 0, g3[16 / 25], g3[1]}
```

{-1, 1, 1}

```
Sign /@ {f3[x] /. x → 0, f3[16 / 25], f3[1]}
```

{-1, -1, 1}

---

```
Sign /@ {g2[0], g2[33 / 100], g2[1 / 2], g2[43 / 50], g2[1]}
```

{1, -1, -1, 1, 1}

```
Sign /@ {f2[0], f2[33 / 100], f2[1 / 2], f2[43 / 50], f2[1]}
```

{1, 1, -1, -1, 1}

---

```
Sign /@ {g1[0], g1[3 / 5], g1[1]}
```

{0, -1, 0}

```
Sign /@ {f1[0], f1[3 / 5], f1[1] /. sub // Simplify}
```

{0, 1, 0}

---

```
{g0[0], g0[1]}
```

{0, 0}

```
{f0[0], f0[1] /. sub // Simplify}
```

{0, 0}

---

```
{r2[0] // Simplify, r2[0] // N}
```

$\left\{-\dfrac{19}{3} + \dfrac{8}{\sqrt{3}} + \dfrac{7\pi^2}{36}, 0.204559\right\}$

```
{r2[1] // Simplify, r2[1] // N}
```

$\left\{-\dfrac{9}{4} + \dfrac{33}{4\sqrt{5}} - \dfrac{\pi^2}{9}, 0.342889\right\}$

# RS1: Monotonicity of $q_{R,S;1}$ on $(0, 1)$

Following are numerical calculations supporting the arguments of the proof of statement (RS1) in Theorem 2.1.

```
{f[1], g[1], f'[1], g'[1]} /. sub // Simplify
```

$\{0, 0, 0, 0\}$

```
{a, b, c} = {1, f''[1]/g''[1], (f'''[1] g''[1] - f''[1] g'''[1])/(3 g''[1]^2)} /. sub // Simplify
```

$\left\{1, -\frac{9}{4} + \frac{33}{4\sqrt{5}}, 3 - \frac{4}{\sqrt{5}}\right\}$

---

```
Sign /@ {g4[0], g4[4/5], g4[1]}
```

$\{-1, 1, 1\}$

```
Sign /@ {f4[0], f4[4/5], f4[1]}
```

$\{-1, -1, 1\}$

---

```
Sign /@ {Limit[g3[x], x → 0], g3[1/2], g3[1]}
```

$\{1, -1, 0\}$

```
Sign /@ {Limit[f3[x], x → 0], f3[1/2], f3[1] // Simplify}
```

$\{1, 1, 0\}$

---

```
Sign /@ {g2[0], g2[1/10], g2[1]}
```

$\{-1, 1, 0\}$

```
Sign /@ {f2[0], f2[1/10], f2[1] // Simplify}
```

$\{-1, -1, 0\}$

---

```
Sign /@ {g1[0], g1[1]}
```

$\{-1, 0\}$

```
Sign /@ {f1[0], f1[1] /. sub // Simplify}
```

$\{-1, 0\}$

```
Sign[-(r2[0] g1[0] - f1[0])]   (* ρ̃₁(0+)>0 *)
```

$1$

---

```
{g0[1], f0[1]}
```

$\{0, 0\}$

---

```
{r0[0] // Simplify, r0[0] // N}
```

$\left\{\frac{7}{20}\left(15 - 7\sqrt{5}\right) + \frac{\pi^2}{9}, 0.868256\right\}$

```
{r4[1] // Simplify, N[r4[1], 8]}   (* r₀(1-)=r₄(1-) by l'Hospital *)
```

$\left\{-\frac{7}{4} + \frac{987}{100\sqrt{5}}, 2.6639982\right\}$



\end{document}